\documentclass[preprint,11pt]{elsarticle}

\usepackage{hyperref}
\usepackage{amssymb}
\usepackage{amsmath}
\usepackage{mathtools}
\usepackage{amsfonts}
\usepackage{amsthm}
\usepackage{times}
\usepackage{tikz-cd}
\usetikzlibrary{arrows}
\usepackage{enumitem} 
\usepackage{graphicx} 
\usepackage{mathrsfs} %for \mathscr
\usepackage{mathtools} 
\usepackage{cleveref} % references 

\theoremstyle{plain}
\newtheorem{thm}{Theorem}[section]
\newtheorem{lemma}[thm]{Lemma}
\newtheorem{cor}[thm]{Corollary}
\newtheorem{prop}[thm]{Proposition}

\theoremstyle{definition}
 
\newtheorem{remark}[thm]{Remark}
\newtheorem{exmp}[thm]{Example}

\newcommand{\set}[1]{\{ #1 \}}  %Set brackets
\newcommand{\alg}[1]{\langle #1 \rangle} %Algebra
\newcommand{\pair}[1]{\langle #1 \rangle} %Pair
\DeclarePairedDelimiterX{\abs}[1]{\lvert}{\rvert}{#1} %Cardinality
\newcommand{\meet}{\wedge} %meet
\newcommand{\join}{\vee} %join

\newcommand{\bigjoin}{\bigvee} %big join
\newcommand{\bigmeet}{\bigwedge} %big meet
\renewcommand{\phi}{\varphi} %\varphi for \phi
\newcommand{\var}[1]{\mathsf{#1}} %Varieties and classes of algebras
\newcommand{\DLM}{\var{DLM}} %distributive l-monoids
\newcommand{\Id}{\var{Id}\DLM} %idempotent distributive l-monoids
\newcommand{\Com}{\var{C}} %commutative distributive l-monoids
\newcommand{\Sem}{\var{Sem}} %semilinear distributive l-monoids
\newcommand{\OM}{\var{OM}} %ordered monoids
\newcommand{\IdOM}{\var{IdOM}} %idempotent ordered monoids
\newcommand{\relD}{\mathrel{\mathscr{D}}} %green relation
\DeclareMathOperator{\setminuss}{\backslash} %fix spacing for \setminus
\renewcommand{\setminus}{\setminuss} %fix spacing for \setminus
\newcommand{\Con}{\mathbf{Con}} %congruence lattice
\newcommand{\con}{\mathrm{Con}} %universe of congruence lattice
\newcommand{\End}{\mathrm{End}} %Endomorphisms
\newcommand{\Sg}{\mathbf{Sg}} %subalgebra
\newcommand{\sg}{\mathrm{Sg}} %universe of subalgebra 
\makeatletter
\newcommand*{\bigboxplus}{%
  \DOTSB
  \mathop{\vphantom{\bigoplus}\mathpalette\matt@bigboxplus\relax}%
  \slimits@
}
\newcommand\matt@bigboxplus[2]{%
  \vcenter{\m@th\hbox{\resizebox{\widthof{$#1\bigoplus$}}{!}{$\boxplus$}}}%
}
\makeatother
\journal{}

\begin{document}

\begin{frontmatter}

\title{Semilinear idempotent  distributive \texorpdfstring{$\ell$}{l}-monoids}

\author[1]{Simon Santschi%
\fnref{fn1}}
\ead{simon.santschi@unibe.ch}
\affiliation[1]{organization= {Mathematical Institute, University of Bern}, addressline={Alpeneggstrasse 22},
postcode={3012}, city={Bern}, country={Switzerland}}
\fntext[fn1]{Supported by Swiss National Science Foundation grant 200021\textunderscore 184693.}

\begin{abstract}
We prove a representation theorem for totally ordered idempotent monoids via a nested sum construction. Using this representation theorem we obtain a characterization of the subdirectly irreducible members  of the variety of semilinear idempotent distributive $\ell$-monoids and a proof that its lattice of subvarieties is countably infinite. For the variety of commutative idempotent distributive $\ell$-monoids we give an explicit description of its lattice of subvarieties and show that each of its subvarieties is finitely axiomatized. Finally we give a characterization of which spans of totally ordered idempotent monoids have an amalgam in the class of totally ordered monoids, showing in particular that the class of totally ordered commutative idempotent monoids has the strong amalgamation property and that various classes of distributive $\ell$-monoids do not have the amalgamation property.  We also show that exactly seven non-trivial finitely generated subvarieties of the variety of semilinear idempotent distributive $\ell$-monoids have the amalgamation property; we are able to determine for all but three of its  subvarieties  whether they have the amalgamation property or not.
\end{abstract}


\end{frontmatter}
 
%%%%%%%%%%%%%%%%%%%%%%%%%%%%%%%%%%%%%%%%%%%%%%%%%%%%%%%%%

\section{Introduction}\label{sec1}

Distributive $\ell$-monoids are monoids with a distributive lattice-order such that multiplication distributes over both binary meets and binary joins. They occur naturally as inverse-free reducts of lattice-ordered groups ($\ell$-groups)  and indeed have the same equational theory as the latter \cite{colacito2021}. More generally, distributive $\ell$-monoids occur as residual-free reducts of fully distributive residuated lattices (see e.g., \cite{Cardona2017}) and a systematic study of distributive $\ell$-monoids can further the investigation of these reducts and the corresponding implication-free fragments of substructural logics \cite{Galatos2007}  (see e.g, \cite{GarciaCerdana2013}). For example, it can help  answer the question whether two classes of fully distributive residuated  lattices can be distinguished by residual-free equations or quasi-equations.

Despite providing one of the most natural generalizations of $\ell$-groups, the structure theory of distributive $\ell$-monoids is not yet at a very sophisticated state, possibly because the tools and techniques of $\ell$-group theory do not extend well in the absence of inverses.
In particular, there does not yet exist  a good characterization of congruences of distributive $\ell$-monoids (but see \cite{Merlier1971,Bosbach1991,colacito2021}, where congruences are constructed from prime lattice ideals of distributive $\ell$-monoids) and not much is known about subdirectly irreducible members of this class to date (see \cite{Rep84}, covering the finite commutative case).
On the other hand, there exists a natural analogue of the representation theorem of Holland  \cite{Holland1963} of $\ell$-groups as subalgebras of  $\ell$-groups of automorphisms of chains. That is, in \cite{anderson_edwards_1984} Anderson and Edwards show that every distributive $\ell$-monoid embeds into the distributive $\ell$-monoid of endomorphisms of a chain. 
Note, however, that in \cite{Rep83} Repnitskii shows that the variety generated by the inverse-free reducts of abelian $\ell$-groups is not finitely axiomatizable, gives a recursive axiomatization for this variety, and proves that in contrast to the case of all distributive $\ell$-monoids, there are equations that hold in this variety but not in all commutative distributive $\ell$-monoids. Extending the result of Repnitskii, it is proved in \cite{colacito2021} that there are equations that hold in all inverse-free reducts of totally ordered $\ell$-groups that do not hold in all totally ordered distributive $\ell$-monoids.

In this paper, we consider the variety of semilinear idempotent distributive $\ell$-monoids, i.e., distributive $\ell$-monoids that are subdirect products of totally ordered  idempotent monoids. 
At the center of this paper lies the nested sum construction. 
The nested sum construction was considered by Aglianò and Montagna in \cite{Agliano2003} under the name `ordinal sum'  for totally ordered BL-algebras. It was generalized by Galatos in \cite{Galatos2004}. The nested sum construction has been used to obtain representation results in a range of settings, including totally ordered BL-algebras \cite{Agliano2003}, totally ordered $n$-contractive MTL-algebras \cite{Horcik2007},  semilinear commutative idempotent  residuated lattices \cite{Olson2008,Olson2012}, and \mbox{$^\star$-involutive} idempotent residuated chains \cite{Fussner2022a}.

In the current paper, we extend on one hand the theory of nested sum representations to totally ordered idempotent monoids, and on the other hand we use the nested sum construction to overcome some of the current deficits of the theory of distributive $\ell$-monoids. 
In particular, we give an explicit characterization of subdirect irreducibility in the variety of semilinear idempotent distributive $\ell$-monoids in terms of the nested sum representation which also leads to a better understanding and description of the subvariety lattices.
The nested sum representation is not only useful to further the structure theory, but also helps us better understand homomorphisms between totally ordered idempotent monoids. In particular, it is useful for 
establishing the amalgamation property for classes of idempotent distributive $\ell$-monoids, a fundamental algebraic property (see e.g., \cite{Kiss1983}) that has been studied in great depth for many varieties of  residuated lattices (see \cite{Metcalfe2014} for an overview and further references).

The paper is structured as follows. 
In \Cref{sec2}, we introduce the necessary definitions and preliminary results about distributive $\ell$-monoids that will be used in the later sections. 
In \Cref{sec3}, we lay the ground work for the paper. We first introduce the nested sum construction for the finite case and prove a representation theorem (\Cref{l: idemp ord monoid e-sum}) for finite totally ordered idempotent monoids in terms of nested sums, followed by some results obtained using this representation. Then we introduce the nested sum construction in the general case and prove a representation theorem (\Cref{thm:general e-sum decomposition}) for arbitrary totally ordered idempotent monoids.  Using the general representation theorem we characterize embeddings between totally ordered idempotent monoids with respect to their nested sum decomposition (\Cref{l:embeddings to summands}).

In \Cref{sec4}, we provide a characterization of the subdirectly irreducible members of the variety of semilinear idempotent distributive $\ell$-monoids in terms of nested sums first  for the finite case (\Cref{thm:subdirect-equivalence})  and then for the general case  (\Cref{thm:general subd equivalence}). The characterization in the finite case also yields a counting result for finite subdirectly irreducibles (\Cref{cor:counting subd}).

In \Cref{sec5}, we use the characterization of finite subdirectly irreducibles  to investigate the lattice of subvarieties of the variety of idempotent distributive $\ell$-monoids. Similarly to  \cite{Olson2012} we use  the nested sum representation to obtain a description of the lattice of subvarieties of the variety of semilinear idempotent distributive $\ell$-monoids, resulting in a proof that the lattice of subvarieties is countably infinite (\Cref{thm:subvariety lattice}). 

In \Cref{sec6}, we first specialize the results of \Cref{sec4} and \Cref{sec5} to the commutative case and then obtain an explicit description of the lattice of subvarieties of the variety of commutative idempotent distributive $\ell$-monoids (\Cref{thm:CId lattice}). Using the explicit description of its lattice of subvarieties we establish that every proper subvariety of the variety of commutative idempotent distributive $\ell$-monoids can be relatively axiomatized by a single equation (\Cref{thm:axiomatizations CId}). We conclude the section with a  brief discussion of some consequences of the results in the context of Sugihara monoids and commutative fully distributive residuated lattices. 

In \Cref{sec7}, we shift our focus to the amalgamation property for classes of semilinear idempotent distributive $\ell$-monoids. First, we use the nested sum representation to give a complete characterization of the spans of totally ordered idempotent monoids that have an amalgam in the class of totally ordered monoids  (\Cref{thm:amalgamation of chains}) in terms of not `restricting' to one of two forbidden spans. In particular, we show  that the class of totally ordered commutative idempotent monoids has the strong amalgamation property (\Cref{c:com amalgamation}) and the forbidden spans enable us to show that several classes of distributive $\ell$-monoids do  not have the amalgamation property. 
Finally, we show that exactly seven non-trivial finitely generated subvarieties of the variety of semilinear idempotent distributive $\ell$-monoids have the amalgamation property (\Cref{thm:subvariety amalgamation}), and establish the failure of this property for all but three non-finitely generated subvarieties (\Cref{c:amalgamation infinite}).

\section{Preliminaries}\label{sec2}

\paragraph{Conventions}
We assume that the reader is familiar with the basic notions of universal algebra which can be found for example in \cite{Burris_Sankappanavar_1981}. We mostly follow the notation of \cite{Burris_Sankappanavar_1981}. In particular  for an algebra $\bf A$ we denote its universe by $A$, for $S\subseteq A$ we denote the subalgebra of $\bf A$ generated by $S$ by $\Sg(S)$,  the congruence lattice of $\bf A$ by $\Con(\mathbf{A})$, the trivial congruence on $\bf A$ by  $\Delta_{ A}$, i.e., $\Delta_{A} = \set{\pair{a,a} \mid a \in A}$, and for $a,b\in A$, the principal congruence generated by $\set{\pair{a,b}}$ by $\Theta^{\bf A}(a,b)$. We write $\Theta(a,b)$ if the algebra is clear. 
For a class of algebras $\var{K}$ we denote by $H(\var{K})$, $S(\var{K})$, $P(\var{K})$, $P_U(\var{K})$, $P_S(\var{K})$, and $I(\var{K})$ the closure of $\var{K}$ under homomorphic images, subalgebras, products, ultraproducts, subdirect products, and isomorphic images, respectively. Moreover, we will denote by $V(\var{K})$ and $Q(\var{K})$ the variety and quasivariety generated by~$\var{K}$, respectively.  For a variety $\var{V}$ and a set $X$, we denote by $\mathbf{F}_\var{V}(X)$ the $\var{V}$-free algebra over $X$ and for $n\in \mathbb{N}$ we write $\mathbf{F}_\var{V}(n)$ for $\mathbf{F}_\var{V}(\set{1,\dots,n})$.

A \emph{distributive $\ell$-monoid}  is an algebra $\mathbf{M} = \alg{M,\meet,\join,\cdot,e}$ such that 
\begin{enumerate}[label = {(\arabic*)}]
\item $\alg{M,\meet,\join}$ is a distributive lattice,
\item $\alg{M,\cdot,e}$ is a monoid,
\item for all $a,b,c,d \in M$,
\begin{equation*}
a(b\join c) d = abd \join acd \quad \text{and} \quad a(b\meet c) d = abd \meet acd.
\end{equation*}
\end{enumerate}
As in (3) we will sometimes  write $ab$ for $a\cdot b$, we will drop unnecessary brackets if no confusion arises, and we will  assume that $\cdot$ binds stronger than $\meet$ and $\join$.

The class $\DLM$ of distributive $\ell$-monoids forms a variety (equational class). We call a distributive $\ell$-monoid \emph{idempotent} if its monoid reduct is idempotent, i.e., satisfies the equation $x^2 \approx x$, and \emph{commutative} if its monoid reduct is commutative. We denote the variety of idempotent distributive $\ell$-monoids by $\Id$ and the variety of commutative idempotent distributive $\ell$-monoids by $\Com\Id$.

\begin{exmp}
An \emph{$\ell$-group} is an algebra $\mathbf{L} = \alg{L,\meet,\join,\cdot,{}^{-1},e}$ such that $\alg{L,\meet,\join}$ is a lattice, $\alg{L,\cdot,{}^{-1},e}$ is a group, and $a\leq b$ implies $cad \leq cbd$ for all $a,b,c,d\in L$, where $\leq$ is the lattice order of $\bf L$. It is well-known that the lattice reduct of an $\ell$-group is distributive and that products distribute over meets and joins. Hence, the inverse-free reducts of $\ell$-groups are distributive $\ell$-monoids.  
\end{exmp}

\begin{exmp}
Let $\alg{\Omega,\leq}$ be a chain, i.e., a totally ordered set.
Then the set $\textnormal{End}(\alg{\Omega,\leq})$ of all order-preserving endomorphisms on $\Omega$ with composition~$\circ$ and point-wise lattice-order gives rise to the distributive $\ell$-monoid $\mathbf{End}(\alg{\Omega,\leq}) = \alg{\textnormal{End}(\Omega), \meet,\join, \circ, id_\Omega}$.
\end{exmp}
In fact every distributive $\ell$-monoid can be seen as a subalgebra of an $\ell$-monoid of  endomorphisms on a chain.
Similarly to $\ell$-groups (see \cite{Holland1963}) there is a Holland-style representation theorem for distributive $\ell$-monoids.
\begin{thm}[\cite{anderson_edwards_1984}, see also \cite{Bosbach1988}]\label{thm:representation}
Every distributive $\ell$-monoid embeds into the $\ell$-monoid $\mathbf{End}(\alg{\Omega,\leq})$ of endomorphisms on some chain $\alg{\Omega,\leq}$.
\end{thm}

We call a distributive $\ell$-monoid $\mathbf{M} = \alg{M,\meet,\join,\cdot,e}$ an \emph{ordered monoid} if its lattice order is a total order. As there is a one-to-one correspondence, we will also consider ordered monoids as relational structures $\mathbf{M} = \alg{M,\cdot, e,\leq}$, where $\alg{M,\cdot,e}$ is a monoid, $\leq$ is a total order on $M$,  and for all $a,b,c,d\in M$, $a\leq b$ implies $cad \leq cbd$. We note that  a map between ordered monoids is a homomorphism if and only if it is a monoid homomorphism and it is order-preserving. We denote the class of ordered monoids by $\OM$ and the class of idempotent ordered monoids by $\IdOM$. We will also denote the class of commutative idempotent ordered monoids by $\Com\IdOM$.

We call a distributive $\ell$-monoid \emph{semilinear}  if it is contained in the variety $\Sem\DLM$ generated by the class $\OM$. We denote the variety of semilinear idempotent  distributive $\ell$-monoids by $\Sem\Id$. Note that in the literature semilinear distributive $\ell$-monoids are also called \emph{representable} (see e.g., \cite{colacito2021}) following the nomenclature for $\ell$-groups. 

Recall that an algebra $\bf A$ is called \emph{congruence-distributive} if its congruence lattice $\Con(\mathbf{A})$ is a distributive lattice and a class of algebras is called \emph{congruence-distributive} if all of its members are congruence-distributive. It is well-known that lattices are congruence-distributive and hence every algebra with a lattice reduct is also congruence-distributive. In particular, every distributive $\ell$-monoid is congruence-distributive. We also recall the following useful result by Jónsson about subdirectly irreducible algebras in congruence-distributive varieties:
\begin{thm}[Jónsson's Lemma {\cite[Theorem 3.3.]{Jonsson1967}}]\label{thm:jonsson}
Let $\var{K}$ be a class of algebras such that $V(\var{K})$ is congruence-distributive. Then every subdirectly irreducible algebra in $V(\var{K})$ is contained in $HSP_U(\var{K})$ and $V(\var{K}) = IP_SHSP_U(\var{K})$.
\end{thm}
In this paper we will often use a special case of \Cref{thm:jonsson}.
\begin{cor}[{\cite[Corollary 3.4.]{Jonsson1967}}]\label{l:jonsson}
If $\var{K}$ is a finite set of  finite algebras such that  $V(\var{K})$ is congruence-distributive,  then the subdirectly irreducible algebras of $V(\var{K})$ are in $HS(\var{K})$, and $V(\var{K}) = IP_S(HS(\var{K}))$.
\end{cor}

The class  $\OM$ of ordered monoids is a positive universal class, i.e., closed under $HSP_U$. So, by \Cref{thm:jonsson}, we get the following result.
\begin{prop}\label{cor:subdirect irred in Sem}\leavevmode
Every subdirectly irreducible algebra in $\Sem\DLM$ is totally ordered and a distributive $\ell$-monoid is semilinear if and only if it is isomorphic to a subdirect product of ordered monoids.
\end{prop}
It follows, in particular, that $\Sem\Id$ is the variety generated by $\IdOM$.
We have the following equational characterization of semilinearity, where for  terms $s$ and $t$ we denote by $s\leq t$ the equation $s\meet t \approx s$.

\begin{prop}[{\cite[Proposition 5.4.]{colacito2021}}, see also {\cite[Corollary 6.10]{Bosbach1988}}]\label{l:semilinear}
A distributive $\ell$-monoid  is semilinear if and only if  it satisfies the equation $z_1xz_2 \meet w_1yw_2 \leq z_1yz_2 \join w_1xw_2$.
\end{prop}
An important consequence of \Cref{l:semilinear} is the following result, which was first proved by Merlier in \cite{Merlier1971}.
\begin{cor}[{\cite[Corollary 2]{Merlier1971}}, {\normalfont see also} \cite{colacito2021}]
Every commutative distributive $\ell$-monoid is semilinear.
\end{cor}
Hence, in particular  $\Com\Id $ is the subvariety of $\Sem\Id$ consisting of its commutative members and is generated by $\Com\IdOM$.

For a distributive $\ell$-monoid $\mathbf{M} = \pair{M,\meet,\join,\cdot, e}$, we define its \emph{order dual}~$\mathbf{M}^\partial$ by  $\mathbf{M}^\partial = \pair{M,\join,\meet,\cdot, e}$. Taking the order dual of a distributive  $\ell$-monoid corresponds to reversing its lattice order, so in particular for an ordered monoid $\mathbf{M} = \alg{M,\cdot,e,{\leq}}$ we have  $\mathbf{M}^\partial = \alg{M,\cdot,e,{\geq}}$. For a term $t$ we define the \emph{dual}~$t^\partial$  of $t$ recursively by
\begin{itemize}
\item $x^\partial = x$ if $x$ is a variable;
\item $(u\cdot v)^\partial = u^\partial\cdot v^\partial$;
\item $(u\meet v)^\partial = u^\partial \join v^\partial$;
\item $(u\join v)^\partial = u^\partial \meet v^\partial$;
\end{itemize}
and we extend the notion of a dual to equations by setting $(t\approx s)^\partial \coloneqq t^\partial \approx s^\partial$. 

For a variety $\var{V}$ of distributive $\ell$-monoids we denote by $\var{V}^\partial$ the class consisting of the order duals of the algebras in $\var{V}$. It is clear that $\var{V}^\partial$ is also a variety axiomatized by the duals of the axioms of $\var{V}$. We call a variety $\var{V}$ \emph{self-dual} if $\var{V} = \var{V}^\partial$. 

\begin{lemma}\label{l:self-dual}
Let $\var{V}$ be a variety of distributive $\ell$-monoids axiomatized by a set~$\Sigma$ of equations. Then $\var{V}$ is self-dual if and only if $\var{V} \models \varepsilon^\partial$ for every  $\varepsilon \in \Sigma$.
\end{lemma}
So we obtain the following corollary, where we note that the dual of the equation  from \Cref{l:semilinear} is a substitution instance of the original equation. 
\begin{cor}
 The varieties $\DLM$,  $\Id$, $\Com\Id$,   and  $\Sem\Id$ are self-dual. 
\end{cor}

Most of the varieties we consider in this paper are self-dual, so in the following we will use this fact to shorten some of the proofs by referring to duality. 

A variety $\var{V}$ is called \emph{locally finite} if every finitely generated algebra contained in $\var{V}$ is finite. For a variety $\var{V}$ of distributive $\ell$-monoids we denote by $\var{V}_\mathsf{m}$ the variety generated by the monoid reducts of $\var{V}$. Note that $\var{V}_\mathsf{m}$ is a variety of monoids.

\begin{lemma}\label{lemma:locally finite DLM}
A subvariety $\var{V}$ of $\DLM$ is locally finite if and only if the variety~$\var{V}_\mathsf{m}$ generated by the monoid reducts of the members of $\var{V}$ is locally finite.
\begin{proof}
First note that a variety is locally finite if and only if every finitely generated free algebra is finite. 

For the right-to-left direction assume that  $\var{V}_\mathsf{m}$ is locally finite. Then for every $n\in \mathbb{N}$ the free algebra $\mathbf{F}_{\var{V}_\mathsf{m}}(n)$ is finite and, using the distributivity of the lattice reduct and the distributivity of products over meets and joins, it is straightforward to see that every member of the free algebra $\mathbf{F}_\var{V}(n)$ corresponds to a meet of joins of members of  $\mathbf{F}_{\var{V}_\mathsf{m}}(n)$. Thus, since $\mathbf{F}_{\var{V}_\mathsf{m}}(n)$ is finite,  modulo the lattice axioms there are only finitely many different such terms. Hence also $\mathbf{F}_\var{V}(n)$ has to be finite for every $n\in \mathbb{N}$, i.e., $\var{V}$ is locally finite.

For the left-to-right direction note first that, by the definition of $\var{V}_\mathsf{m}$, for any monoid terms $s,t$  we have $\var{V} \models s\approx t$ if and only if $\var{V}_\mathsf{m} \models s \approx t$, so  $F_{\var{V}_\mathsf{m}}(n)$ can be  considered as a subset of $F_\var{V}(n)$. Now suppose that $\var{V}$ is locally finite. Then for every $n\in \mathbb{N}$ the free algebra  $\mathbf{F}_{\var{V}}(n)$ is finite and thus also  $\mathbf{F}_{\var{V}_\mathsf{m}}(n)$ is finite. Hence $\var{V}_\mathsf{m}$ is locally finite.
\end{proof}
\end{lemma}

\begin{thm}[\cite{green_rees_1952},  {\cite[Theorem 2]{McLean1954}}]
The variety of idempotent monoids is locally finite.
\end{thm}
Hence, by \Cref{lemma:locally finite DLM}, we get some immediate corollaries.
\begin{cor}
Every variety of idempotent distributive $\ell$-monoids is locally finite.
\end{cor}
Since every variety of  idempotent distributive $\ell$-monoids is locally finite, recursive axiomatizability yields  the decidability of their  quasi-equational theories.
\begin{cor}
Every recursively axiomatized subvariety of $\Id$ has a decidable quasi-equational theory. In particular the variety $\Sem\Id$ has a decidable quasi-equational theory.
\end{cor}
\begin{cor}
Every subvariety of $\Sem\Id$ is generated by its finite subdirectly irreducible (totally ordered) members as a quasivariety.
\end{cor}

Thus to describe subvarieties of $\Sem\Id$, it is enough to consider the finite subdirectly irreducible members of a given subvariety.

\section{Nested sums}\label{sec3}

For idempotent distributive $\ell$-monoids the product is much more restricted than in the general non-idempotent case.

\begin{lemma}[{\normalfont cf.} \cite{Merlier1981,Stanovsky2007}]\label{lemma:Idemp}
For every idempotent distributive $\ell$-monoid  $\bf M$  and $a,b \in M$ we have:
\begin{enumerate}[label = \textup{(\roman*)}]
\item $a\meet b \leq ab \leq a\join b$.
\item If $e\leq ab$, then $ab = a\join b$.
\item If $ab\leq e$, then $ab = a \meet b$.
\item If $\bf M$ is totally ordered, then $ab \in \set{a,b}$.
\end{enumerate}
\end{lemma}
Note that it follows from \Cref{lemma:Idemp} that for every idempotent ordered monoid $\bf M$ and subset $S\subseteq M$, the set $S\cup \set{e}$ is the universe of the subalgebra $\Sg(S)$ of $\bf M$ and $ab \in \set{a,b}$ for all $a,b\in S \cup \set{e}$. In the sequel we will use this fact without explicitly mentioning it.

\begin{lemma}\label{l:top-bot-cases}
Let $\mathbf{M}$ be a non-trivial idempotent ordered monoid with top element $\top$ and bottom element $\bot$. Then exactly one of the following holds:
\begin{enumerate}[label = \textup{(\arabic*)}]
\item  For all $a\in M$ we have  $\bot \cdot a = a\cdot \bot = \bot$.
\item For all $a\in M$ we have $\top \cdot a = a\cdot \top = \top$.
\item  For all $a\in M\setminus \set{\top,\bot}$ we have  $\bot \cdot a = a\cdot \bot = \bot$, $\top \cdot a = a\cdot \top = \top$,   $\bot \cdot \top = \bot$, and  $\top \cdot \bot = \top$.
\item For all $a\in M\setminus \set{\top,\bot}$  we have $\bot \cdot a = a\cdot \bot = \bot$, $\top \cdot a = a\cdot \top = \top$,   $\bot \cdot \top = \top$, and  $\top \cdot \bot = \bot$.
\end{enumerate}
\end{lemma}

\begin{proof}
First note that if $\bot \cdot \top = \top \cdot \bot = \bot$, then  we get $\bot \cdot a = a\cdot \bot = \bot$ for all $a\in M$, since $\bot \cdot a \leq \bot \cdot \top$ and $a\cdot \bot \leq \top \cdot \bot$. Similarly if $\top \cdot \bot = \bot \cdot \top = \top$, then  we get $\top \cdot a = a\cdot \top = \top$ for all $a\in M$. Suppose that $\bot \cdot \top = \bot$ and $\top \cdot \bot = \top$ and let $a\in M\setminus\set{\bot,\top}$. We show $\bot \cdot a = a \cdot \bot = \bot$.
If $\bot <a\leq e$, then $\bot \cdot a = a\cdot \bot = \bot$, by \Cref{lemma:Idemp}. So, since $\mathbf{M}$ is totally ordered it suffices to consider the case $e \leq a < \top$. Then, by \Cref{lemma:Idemp}, $a\cdot \top = \top \cdot a = \top$. Hence, $\bot \cdot a = \bot \cdot \top \cdot a = \bot \cdot \top = \bot$. On the other hand if $a\cdot \bot = a$, then $a\cdot \top = a \cdot \bot \cdot \top = a\cdot \bot = a$, contradicting the fact that $a\cdot \top = \top$. Thus, also $a\cdot \bot = \bot$. That $\top \cdot a = a\cdot \top = \top$ for all $a\in M \setminus\set{\bot,\top}$ follows by duality. Moreover, the case where  $\bot \cdot \top = \top$ and $\top \cdot \bot = \bot$ is symmetrical. Finally, exactly one of the four cases holds, since in every non-trivial idempotent ordered monoid we have $\bot\neq \top$.
\end{proof}

\begin{exmp}\label{ex:simple chains}
Consider the algebras $\mathbf{C}_2$, $\mathbf{C}_2^\partial$, $\mathbf{G}_3$, and $\mathbf{D}_3$ which are defined as follows:
\begin{center}
\includegraphics{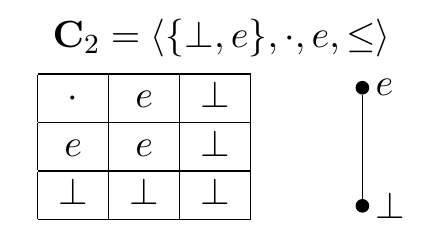}
\includegraphics{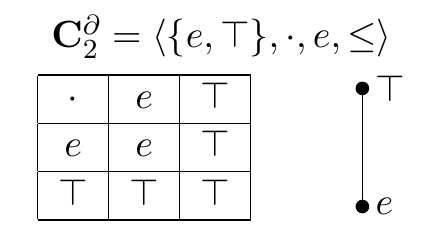}
\includegraphics{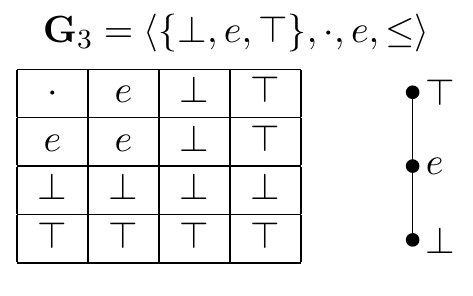}
\includegraphics{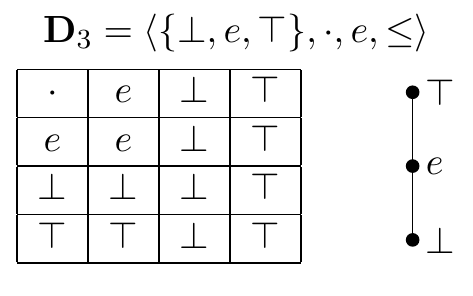}
\end{center}

The algebras $\mathbf{C}_2$, $\mathbf{C}_2^\partial$, $\mathbf{G}_3$, and $\mathbf{D}_3$ are idempotent ordered monoids. The notation of $\mathbf{C}_2^\partial$ makes sense, since it is isomorphic to the order dual of $\mathbf{C}_2$. The notation for $\mathbf{C}_2$ and $\mathbf{C}_2^\partial$ also indicates that the algebras are commutative and have two elements. The notation for  $\mathbf{G}_3$ indicates that it has three elements and that the non-identity elements are left-absorbing (French: `à gauche') and similarly $\mathbf{D}_3$ has three elements and the non-identity elements are right-absorbing (French: `à droite').
Also, the algebras $\mathbf{G}_3$ and $\mathbf{D}_3$ are connected. Indeed, we have $\mathbf{G}_3 \cong \alg{D_3,\ast, e, \leq}$, where $a\ast b := b\cdot a$.  Moreover, $\mathbf{C}_2$ and $\mathbf{C}_2^\partial$ are subalgebras of $\mathbf{G}_3$ and $\mathbf{D}_3$ and it is easy to see that  $\mathbf{C}_2$, $\mathbf{C}_2^\partial$, $\mathbf{G}_3$, and $\mathbf{D}_3$ are simple. Indeed, already their monoid reducts are simple. We also note that each of these four algebras corresponds to a different case of \Cref{l:top-bot-cases}. They are the minimal examples of each case. 
\end{exmp}
As discussed in \Cref{sec1}, nested sums were introduced by Galatos in \cite{Galatos2004} and since used to provide a plethora of structure theorems for classes of residuated lattices (see \cite{Agliano2003,Horcik2007,Olson2008,Olson2012}). They have been deployed especially heavily in the context of BL-algebras, where they are usually called ordinal sums \cite{Agliano2003}. Following the discussion in Fussner and Galatos \cite{Fussner2022a}, we adopt the terminology nested sum instead of ordinal sum.

Let $\bf M$ and $\bf N$ be idempotent ordered monoids, where we relabel the elements such that $M\cap N = \set{e}$.
We define the \emph{nested sum} of $\bf M $ and $\bf N$ by  $\mathbf{M} \boxplus \mathbf{N} = \alg{M\cup N, \cdot, e \leq}$, where $\cdot$ is the extension of the monoid operations on $\bf M$ and $\bf N$ with $a\cdot b= b\cdot a = a$ for all $a\in M\setminus\set{e}$ and $b\in N$, and $\leq$ is the least extension of  the orders of $\bf M$ and $\bf N$ that satisfies for all $a\in M\setminus\set{e}$ and $b\in N$ that $a\leq b$ if $a\leq_\mathbf{M} e$ and $b\leq a $ if $e \leq_\mathbf{M} a$. 

Intuitively this means that we replace the identity $e$ in $\bf M$ with $\bf N$ and extend the order and product in such a way that the elements of $\bf N$ behave with respect to elements of $\bf M$ like $e$.

\begin{exmp} 
Consider the algebras $\mathbf{C}_2$ and $\mathbf{G}_3$ of \Cref{ex:simple chains}, where we rename the element $\bot$ in $\mathbf{C}_2$ as $1$. The order of the nested sum $\mathbf{G}_3 \boxplus \mathbf{C}_2$ is drawn in \Cref{fig:nested sum example} and it has the following multiplication table: 
\begin{center}
\begin{tabular}{|c|c|c|c|c|}
	    \hline
         $\cdot$ & $e$&$1$& $\bot$ & $\top$   \\
         \hline
         $e$  &  $e$& $1$& $\bot$ & $\top$   \\
         \hline
         $1$  & $1$ &$1$   & $\bot$ & $\top$ \\
         \hline
         $\bot$ &$\bot$ & $\bot$ & $\bot$ & $\bot$ \\
         \hline
         $\top$ &$\top$ & $\top$ & $\top$ & $\top$ \\
         \hline
    \end{tabular}
\end{center}
It is no coincidence that we consider the algebras of \Cref{ex:simple chains}. We will see later that every finite idempotent ordered monoid is isomorphic to a nested sum of these four algebras. 
\begin{figure}
\centering
\includegraphics{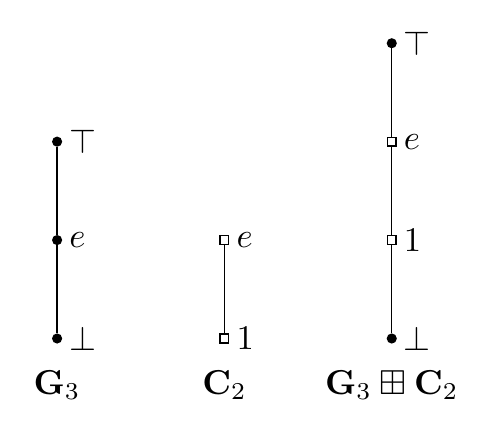}
\caption{The nested sum $\mathbf{G}_3 \boxplus \mathbf{C}_2$} 
\label{fig:nested sum example}
\end{figure}
\end{exmp}

\begin{lemma}[{\mdseries cf.} {\cite[Proposition 4.8]{Olson2008}}]\label{l:e-sum}
Let $\bf M$ and $\bf N$ be idempotent ordered monoids. Then  $\mathbf{M} \boxplus \mathbf{N}$ is an idempotent ordered monoid. Moreover,  $\bf M$ and~$\bf N$ embed  into $\mathbf{M} \boxplus \mathbf{N}$ via the inclusion maps.
\end{lemma}

\begin{lemma}[{\mdseries cf.} {\cite[Proposition 3.3]{Olson2012}}]
Let $\mathbf{M}_1$, $\mathbf{M}_2$, and  $\mathbf{M}_3$ be idempotent ordered monoids. Then
\begin{equation*}
\mathbf{M}_1 \boxplus (\mathbf{M}_2 \boxplus \mathbf{M}_3) \cong (\mathbf{M}_1 \boxplus \mathbf{M}_2) \boxplus \mathbf{M}_3.
\end{equation*}
\end{lemma}

Since the nested sum is associative up to isomorphism, it makes sense to use the notation  $\mathbf{M}_1 \boxplus \cdots \boxplus \mathbf{M}_n$ for nested sums of idempotent ordered monoids $\mathbf{M}_1,\dots, \mathbf{M}_n$. To shorten the notation we will denote this nested sum also by $\bigboxplus_{i=1}^n \mathbf{M}_i$  and we set   $\bigboxplus_{i=1}^0 \mathbf{M}_i = \mathbf{0}$, where $\mathbf{0}$ is a fixed trivial algebra. So we assume that for all nested sums there is a fixed relabeling of the elements.  We note that $\mathbf{M} \boxplus \mathbf{0} \cong \mathbf{0} \boxplus \mathbf{M} \cong \mathbf{M}$ for all $\bf M$, i.e., $\mathbf{0}$ can be seen as the neutral element of the nested sum operation.

\begin{thm}\label{l: idemp ord monoid e-sum}
Each finite idempotent ordered monoid $\bf M$  is isomorphic to a unique nested sum $\bigboxplus^{n}_{i=1}\mathbf{M}_i$ of algebras $\mathbf{M}_i \in \set{\mathbf{C}_2, \mathbf{C}_2^\partial, \mathbf{G}_3,\mathbf{D}_3}$. 
\begin{proof}
We prove the claim by induction on the cardinality of $\bf M$. If $\abs{M} = 1$, then $ \mathbf{M} \cong \bigboxplus^{0}_{i=1}\mathbf{M}_i$. Suppose that $\abs{M} =  m>1$ and  the claim is already proved for all $k<m$. By \Cref{l:top-bot-cases}, there are four cases:
\begin{itemize}[leftmargin =*]
\item If $\bf M$ has a bottom element $\bot$ such that $\bot \cdot a = a\cdot \bot = \bot $ for all $a\in M$, then we have $ \mathbf{M} \cong \mathbf{C}_2 \boxplus \mathbf{L}$, where $\bf L$ is the subalgebra of $\bf M$ with the subuniverse $L = M\setminus \set{\bot}$.
\item If $\bf M$ has a top element $\top$ such that $\top \cdot a =  a\cdot \top =\top $ for all $a\in M$, then we have $ \mathbf{M} \cong \mathbf{C}_2^\partial \boxplus \mathbf{L}$, where $\bf L$ is the subalgebra of $\bf M$ with the subuniverse $L = M\setminus \set{\top}$.
\item If $\bf M$ has bottom and top  elements $\bot$ and $\top$ such that $\bot \cdot \top = \bot$ and $\top \cdot \bot = \top$, then we have $ \mathbf{M} \cong \mathbf{G}_3 \boxplus \mathbf{L}$, where $\bf L$ is the subalgebra of $\bf M$ with the subuniverse $L = M\setminus \set{\bot,\top}$.
\item If $\bf M$ has bottom and top  elements $\bot$ and $\top$ such that $\top \cdot \bot = \bot$ and $\bot \cdot \top = \top$, then we have $ \mathbf{M} \cong \mathbf{D}_3 \boxplus \mathbf{L}$, where $\bf L$ is the subalgebra of $\bf M$ with the subuniverse $L = M\setminus \set{\bot,\top}$. 
\end{itemize}
In each case we have $\abs{L} < m$, so the claim follows by the induction hypothesis. For uniqueness note that in every step the first term of the nested sum is unique, yielding the uniqueness for every $\mathbf{M}_i$.
\end{proof}
\end{thm}
From the uniqueness of the above  nested sum decomposition we obtain the following counting result:

\begin{cor}\label{thm:counting idemp chains}
The number $\mathbf{I}(n)$ of idempotent ordered monoids with $n\in \mathbb{N}\setminus\set{0}$ elements (up to isomorphism) is recursively defined by  $\mathbf{I}(1) = 1$,  $\mathbf{I}(2) = 2$, and
\begin{equation*}
\mathbf{I}(n) = 2\cdot\mathbf{I}(n-1) + 2\cdot \mathbf{I}(n-2) \quad (n>2).
\end{equation*}
Moreover,
\begin{equation*}
\mathbf{I}(n) = \frac{(1+\sqrt{3})^n - (1-\sqrt{3})^n}{2\sqrt{3}}.
\end{equation*}
\end{cor}
\begin{proof}
It is clear that the trivial algebra is the only $1$-element algebra and that $\mathbf{C}_2$ and $\mathbf{C}_2^\partial$ are the only $2$-element idempotent ordered monoids up to isomorphism. Now every idempotent ordered monoid with $n>2$ elements, by \Cref{l: idemp ord monoid e-sum}, is isomorphic to a nested sum $\mathbf{L} \boxplus \mathbf{M}$ where either $\abs{M} = n-1$ and $\mathbf{L} \in \set{\mathbf{C}_2, \mathbf{C}_2^\partial}$ or $\abs{M} = n-2$ and $\mathbf{L} \in \set{\mathbf{G}_3, \mathbf{D}_3}$. So, by the decomposition, it follows that 
$
\mathbf{I}(n) = 2\cdot\mathbf{I}(n-1) + 2\cdot \mathbf{I}(n-2)
$.
A standard calculation then yields the closed formula for $\mathbf{I}(n)$.
\end{proof}
\begin{remark}
The closed formula for $\mathbf{I}(n)$ is the same formula as in \cite{Gil-FerezJipsenMetcalfe2020} for the number of idempotent residuated chains with $n$ elements, but in \cite{Gil-FerezJipsenMetcalfe2020} the number of $n$-element idempotent residuated chains should be  $\mathbf{I}(n-1)$, since every finite idempotent residuated chain is of the form $\mathbf{C}_2 \boxplus \mathbf{M}$, where $\mathbf{M}$ is an idempotent ordered monoid, i.e., the number of idempotent residuated chains with~${n+1}$ elements is equal to the number of  idempotent ordered monoids with $n$ elements. So an alternative proof of \Cref{thm:counting idemp chains} can be obtained by using the result of \cite{Gil-FerezJipsenMetcalfe2020} about idempotent residuated chain.
\end{remark}

\begin{thm}\label{thm:e-sum homomorphisms}
Let $\mathbf{M}_1,\dots, \mathbf{M}_n$, $\mathbf{N}_1,\dots, \mathbf{N}_n$ be idempotent ordered monoids and for $i=1,\dots, n$ let $\phi_i \colon \mathbf{M}_i \to \mathbf{N}_i$ be a homomorphism with $\phi_i^{-1}(\set{e}) = \set{e}$ for $i=1,\dots,n-1$. Then the map $\phi\colon \bigboxplus_{i=1}^n \mathbf{M}_i \to \bigboxplus_{i=1}^n \mathbf{N}_i$ defined by ${\phi\!\!\restriction_{M_i}} = \phi_i$ is a  homomorphism with $\ker(\phi) = \bigcup_{i=1}^n \ker(\phi_i)$.
\end{thm}

\begin{proof}
It suffices to consider the case $n=2$. Note that the map $\phi$ is well defined, since $\phi_1(e) =\phi_2(e) = e$ and $M_1\cap M_2 = \set{e}$.
First we show that $\phi$ is a monoid homomorphism. By definition we have $\phi(e) = e$. Moreover, $\phi_1$ and $\phi_2$ are monoid homomorphisms, so it suffices to take $a\in M_1\setminus\set{e}$ and $b\in M_2$, and therefore $ab = ba = a$.  Since $\phi_1^{-1}(\set{e}) = \set{e}$, we have  $\phi(a) \in  N_1\setminus\set{e}$ and clearly $\phi(b) \in M_2$. So we get $\phi(ab) = \phi(a) = \phi(a)\phi(b)$, by the definition of the nested sum. Hence $\phi$ is a monoid homomorphism. 
To see that $\phi$ is order-preserving it again suffices to consider $a\in M_1\setminus\set{e}$ and $b\in M_2$, since both $\phi_1$ and $\phi_2$ are order-preserving.  If $a\leq b$, then $a\leq_{\mathbf{M}_1} e$. So, since $\phi_1$ is order-preserving and  $\phi_1^{-1}(\set{e}) = \set{e}$, we have $\phi(a) \leq_{\mathbf{N}_1} e$, $\phi(a) \neq e$, and $\phi(b) \in N_2$, yielding $\phi(a) \leq \phi(b)$. Similarly if $b\leq a$, we get $\phi(b) \leq \phi(a)$.  Hence $\phi$ is a homomorphism. That we have $\ker(\phi) = \ker(\phi_1)\cup \ker(\phi_2)$ is immediate, since $\phi_1[L_1] \cap \phi_2[L_2] = \set{e}$ and $\phi_1^{-1}(\set{e}) = \set{e}$.
\end{proof}

\begin{cor}\label{cor:e-sum homomorphisms}
Let $\mathbf{M}_1,\dots, \mathbf{M}_n$, $\mathbf{N}_1,\dots, \mathbf{N}_n$ be idempotent ordered monoids and for $i=1,\dots, n$ let $\phi_i \colon \mathbf{M}_i \to \mathbf{N}_i$ be an $\ell$-monoid embedding. Then the map $\phi\colon \bigboxplus_{i=1}^n \mathbf{M}_i \to \bigboxplus_{i=1}^n \mathbf{N}_i$ defined by ${\phi\!\!\restriction_{M_i}} = \phi_i$ is an $\ell$-monoid embedding.
\end{cor}

Analogous to the case of similar constructions the nested sum can be generalized to infinite nested sums. Let $\mathbf{C} = \alg{C,\sqsubseteq}$ be a chain and let $\set{\mathbf{M}_c}_{c\in C}$ be a family of idempotent ordered monoids $\mathbf{M}_c = \alg{M_c,\cdot_c,e, \leq_c}$, where we relabel the elements such that $M_c \cap M_d = \set{e}$ for all $c\neq d$. We define the \emph{nested sum} $\bigboxplus_{c\in C} \mathbf{M}_c = \alg{\bigcup_{c\in C} M_c,\cdot,e,\leq}$, where 
\begin{itemize}
\item the product $\cdot$ extends $\cdot_c$ for all $c\in C$ such that for all $c,d \in C$ with $c \sqsubset d$ and  $a\in M_c\setminus\set{e}$, $b\in M_d$ we have $a\cdot b = b \cdot a = a$,
\item and the order  $\leq$ is the least extension of the orders $\leq_c$ such that for all $c,d\in C$ with $c\sqsubset d$ and $a \in M_c\setminus \set{e}$, $b\in M_d$ we have $a\leq b$ if $a\leq_c e$ and $b\leq a$ if $e\leq_c a$.
\end{itemize}
We set $\bigboxplus_{c\in \emptyset} \mathbf{M}_c = \mathbf{0}$. Note that for a non-empty subchain $\bf D$ of ~$\bf C$, if we restrict the operations of $\bigboxplus_{c\in C} \mathbf{M}_d$ to $\bigcup_{d\in D} M_d$, we obtain the nested sum $\bigboxplus_{d\in D} \mathbf{M}_d$.  The finite  nested sums that we considered  above  correspond to the cases where $\bf C$ is a finite chain. 
The next lemma generalizes \Cref{l:e-sum}.
\begin{lemma}
Let $\mathbf{C} = \alg{C,\sqsubseteq}$ be a chain and let $\set{\mathbf{M}_c}_{c\in C}$ be a family of idempotent ordered monoids. Then  the nested sum $\bigboxplus_{c\in C} \mathbf{M}_c$ is an idempotent ordered monoid. Moreover, $\mathbf{M}_c$ embeds into $\bigboxplus_{c\in C} \mathbf{M}_c$ via the inclusion map for every $c\in C$.
\end{lemma}
\begin{proof}
Note that all the properties we need to check involve at most three different algebras in the nested sum. Thus the claim follows from the fact that for idempotent ordered monoids $\mathbf{M}_1,\mathbf{M}_2,\mathbf{M}_3$ the nested sum $\mathbf{M}_1\boxplus\mathbf{M}_2 \boxplus \mathbf{M}_3$ is an idempotent ordered monoid.
\end{proof}

Let $\bf M$ be an idempotent monoid. We consider the Green equivalence relation $\relD_{\bf M}$  defined by letting for $a,b\in M$
\[
 a \relD_{\bf M} b :\!\iff aba = a \text{ and } bab = b.
\]
 In \cite{Dubreil-Jacotin1971} Dubreil-Jacotin proved that an idempotent monoid $\bf M$ is totally orderable if and only if  for all $a,b\in M$ we have $ab \in \set{a,b}$ and every equivalence class of $\relD_{\bf M}$ contains at most two elements (see also \cite{Merlier1972,Merlier1976} or  \cite{Saito1974} for a characterization of totally orderable idempotent semigroups). 

\begin{lemma}[\cite{Dubreil-Jacotin1971}]\label{l:dubreil-jacotin}
Let $\bf M$ be an idempotent ordered monoid. For two distinct elements $a,b\in M$ the following are equivalent:
\begin{enumerate}[label = \textup{(\arabic*)}]
\item $a\relD_{\bf M} b$.
\item The subalgebra $\Sg(a,b)$  of $\bf M$  is isomorphic to $\mathbf{G}_3$ or $\mathbf{D}_3$.
\item $ab \neq ba$.
\end{enumerate}
Moreover, every equivalence class of $\relD_{\bf M}$ has at most two elements.
\end{lemma}

\begin{proof}
 Let $a,b\in M$ be two distinct elements.
 
(1) $\Rightarrow$ (3): Suppose contrapositively that $ab = ba$, then we get $aba = baa = ba = bba = bab$. Hence $\pair{a,b}\notin {\relD_{\bf M}}$. 

(3) $\Rightarrow$ (2): Suppose $ab \neq ba$. Then clearly $a\neq e$ and $b\neq e$. So the subalgebra $\Sg(a,b)$ of $\bf M$ has universe $\set{a,b,e}$. Moreover, by \Cref{lemma:Idemp}, since $ab \neq ba$, we get $a< e < b$ or $b< e <a$. So $\Sg(a,b)$ is isomorphic to  $\mathbf{G}_3$ or $\mathbf{D}_3$.

(2) $\Rightarrow$ (1): Suppose that the subalgebra $\Sg(a,b)$ of $\bf M$ is isomorphic to $\mathbf{G}_3$ or $\mathbf{D}_3$. Then clearly $aba = a$ and $bab = b$, so $a\relD_{\bf M} b$.

That every equivalence class of $\relD_{\bf M}$ has at most two elements is immediate, by \Cref{lemma:Idemp} and the fact that for any three distinct elements of $\bf M$ either at least two are positive or at least two are negative, and thus do not generate $\mathbf{G}_3$ or $\mathbf{D}_3$.
\end{proof}

\begin{lemma}\label{l:D-classes commutative}
Let $\bf M$ be an idempotent ordered monoid and let  $a,b,c\in M$ be three distinct elements with $a \relD_{\bf M} b$. Then $ac = ca$ and $bc = cb$, and $ac = a$ if and only if $bc = b$.
\end{lemma}
\begin{proof}
First note that, by \Cref{l:dubreil-jacotin}, every equivalence class of $\relD_{\bf M}$ contains at most two elements, so $\pair{a,c},\pair{b,c}\notin {\relD_{\bf M}}$ and we get $ac = ca$ and $bc = cb$.  Suppose now that $ab = a$ and $ba = b$. Then, if $ac  = a$, we get  $bc = bac = ba = b$. Conversely, if $bc = b$, then we get $ac = abc = ab = a$. 
For the case $ab = b$ and $ba = a$ the proof is symmetric, noting that $ac = ca$ and $bc = cb$.
\end{proof}
The next theorem generalizes \Cref{l: idemp ord monoid e-sum}.
\begin{thm}\label{thm:general e-sum decomposition}
Let $\bf M$ be an idempotent ordered monoid. Then we have $\mathbf{M} \cong \bigboxplus_{c\in C} \mathbf{M}_c$ for some chain $\mathbf{C} = \alg{C,\sqsubseteq}$ and $\mathbf{M}_c  \in \set{\mathbf{C}_2, \mathbf{C}_2^\partial, \mathbf{G}_3,\mathbf{D}_3}$. Moreover, this nested sum is unique over the algebras $\set{\mathbf{C}_2, \mathbf{C}_2^\partial, \mathbf{G}_3,\mathbf{D}_3}$.
\end{thm}

\begin{proof}
Let $C = \set{[a] \mid a\in M\setminus\set{e}}$, where $[a] = \set{b\in M \mid a\relD_{\bf M} b}$ for $a\in M$, and define the relation $\sqsubseteq$ on $C$ by 
\begin{equation*}
[a] \sqsubseteq [b] \iff ab = ba = a \text{ or } a \relD_{\bf M} b.
\end{equation*}
By \Cref{lemma:Idemp} (iv) and \Cref{l:D-classes commutative}, $\sqsubseteq$ is well-defined and  $\mathbf{C} = \alg{C,\sqsubseteq}$ is a chain. Define for $a\in M\setminus\set{e}$, the algebra  $\mathbf{M}_{[a]}$ to be the subalgebra $\Sg([a])$  of ~$\bf M$ with universe  $[a]\cup \set{e}$.  By \Cref{l:dubreil-jacotin}, we get for every $c\in C$ that~$\mathbf{M}_c$ is isomorphic to one of the algebras in  $ \set{\mathbf{C}_2, \mathbf{C}_2^\partial, \mathbf{G}_3,\mathbf{D}_3}$.
Moreover, by construction,  we get  $\mathbf{M} = \bigboxplus_{c\in C} \mathbf{M}_c$. The nested sum is unique over the algebras $\set{\mathbf{C}_2, \mathbf{C}_2^\partial, \mathbf{G}_3,\mathbf{D}_3}$, since, by \Cref{l:dubreil-jacotin}, for every nested sum $\mathbf{N} = \bigboxplus_{d\in D} \mathbf{N}_d$ over the algebras $\set{\mathbf{C}_2, \mathbf{C}_2^\partial, \mathbf{G}_3,\mathbf{D}_3}$ with $\mathbf{M} = \mathbf{N}$ we have for all   $a,b \in N\setminus \set{e}$ that  $a \relD_{\bf M} b$ if and only if $a,b\in N_d$ for some $d\in D$. This also implies that $\mathbf{D}$ and $\mathbf{C}$ are isomorphic.
 \end{proof}

\begin{remark}
Similar constructions also using the Green equivalence relation $\mathscr{D}$ were considered in \cite{Chen2009a, Chen2009b} for idempotent residuated chains and conic idempotent residuated lattices.
\end{remark}

Let $\mathbf{C}$, $\mathbf{D}$ be chains and  $\mathbf{M} = \bigboxplus_{c\in C} \mathbf{M}_c$, $\mathbf{N} = \bigboxplus_{d\in D} \mathbf{N}_d$ idempotent ordered monoids with $\mathbf{M}_c,\mathbf{N}_d\in \set{\mathbf{C}_2, \mathbf{C}_2^\partial, \mathbf{G}_3,\mathbf{D}_3}$. If $\phi\colon \mathbf{M} \to \mathbf{N}$ is an embedding, then for every $c\in C$, $\phi(\mathbf{M}_c)$ is a subalgebra of $\bf N$ and, by \Cref{l:dubreil-jacotin} and \Cref{thm:general e-sum decomposition},  there exists a unique $d\in C$ such that $\phi(M_c) \subseteq N_d$. Thus we get an injective map $f_\phi\colon C \to D$, where $f_\phi(c)$ is the unique element of $D$ such that $\phi(M_c) \subseteq N_{f_\phi(c)}$.

On the other hand if we have an order-embedding $f\colon \mathbf{C} \to \mathbf{D}$ such that for every $c\in C$, $\mathbf{M}_c$ is a subalgebra of $\mathbf{N}_{f(c)}$, then we can define the map $\phi_f\colon M \to N$, where $\phi_f(a) = a \in N_{f(c)}$ for $a\in M_c$. 

\begin{lemma}\label{l:embeddings to summands}
Let $\mathbf{C}$, $\mathbf{D}$ be chains,  $\mathbf{M} = \bigboxplus_{c\in C} \mathbf{M}_c$, $\mathbf{N} = \bigboxplus_{d\in D} \mathbf{N}_d$ idempotent ordered monoids with $\mathbf{M}_c,\mathbf{N}_d\in \set{\mathbf{C}_2, \mathbf{C}_2^\partial, \mathbf{G}_3,\mathbf{D}_3}$, $\phi \colon \mathbf{M} \to \mathbf{N}$ an embedding, and $f\colon \mathbf{C} \to \mathbf{D}$ an order-embedding such that for every $c\in C$,~$\mathbf{M}_c$ is a subalgebra of $\mathbf{N}_{f(c)}$.
\begin{enumerate}[label = \textup{(\arabic*)}]
\item The map $f_\phi\colon C \to D$ is an order-embedding such that for every $c\in C$,  $\mathbf{M}_c$ is a subalgebra of $\mathbf{N}_{f_\phi(c)}$. 
\item The map $\phi_f \colon M \to N$ is an $\ell$-monoid embedding.
\item We have  $\phi_{f_\phi} = \phi$ and  $f_{\phi_f} = f$.
\end{enumerate}
\end{lemma}

\begin{proof}
(1) By definition of the map $f_\phi$, we already know that $f_\phi$ is injective and that~$\mathbf{M}_c$ is a subalgebra of $\mathbf{N}_{f_\phi(c)}$ for each $c\in C$.  So it remains to show that $f_\phi$ is order-preserving. Let $c_1,c_2 \in C$ such that $c_1 \sqsubset c_2$ and suppose for a contradiction that $f(c_2) \sqsubseteq f(c_1)$. Then for  $a\in M_{c_1}\setminus \set{e}$ and $b\in M_{c_2} \setminus \set{e}$ we have $\phi(a) = \phi(a\cdot b) =  \phi(a) \cdot \phi(b) = \phi(b)$. Hence $a=b$ a contradiction. Hence $f_\phi$ is order-preserving.
 
 (2) By \Cref{cor:e-sum homomorphisms} and since $f$ is an order-embedding, if $c_1,c_2 \in C$ with $c_1 \sqsubset c_2$, then the map $\phi_f$ restricted to $\mathbf{M}_{c_1} \boxplus \mathbf{M}_{c_2}$ is an $\ell$-monoid embedding into $\mathbf{N}_{f(c_1)} \boxplus \mathbf{N}_{f(c_2)}$. So, since checking whether $\phi$ is an $\ell$-monoid embedding involves at most two terms of the nested sum, $\phi_f$ is an $\ell$-monoid embedding. 
 
 (3) is immediate from the definition.
\end{proof}

\begin{cor}\label{cor:IS embedding}
Let $\mathbf{M} = \bigboxplus_{i=1}^m \mathbf{M}_i$, $\mathbf{N} = \bigboxplus_{j=1}^n \mathbf{N}_j$ be idempotent ordered monoids with $\mathbf{M}_i,\mathbf{N}_j\in \set{\mathbf{C}_2, \mathbf{C}_2^\partial, \mathbf{G}_3,\mathbf{D}_3}$. Then $\mathbf{M}$ embeds into $\mathbf{N}$ if and only if there exists an order-embedding $f\colon \set{1,\dots, m} \to \set{1,\dots,n}$ such that for every $i\in \set{1,\dots,m}$ we have  that $\mathbf{M}_i$ is a subalgebra of $\mathbf{N}_{f(i)}$.
\end{cor}

\section{Subdirectly irreducible algebras in {\sf SemIdDLM}}\label{sec4}

In this section we use the nested sum decomposition to give a characterization of the subdirectly irreducible members of $\Sem\Id$.
Note that in contrast to $\ell$-groups and residuated lattices not very much is known about congruences and subdirectly irreducible elements in varieties of distributive $\ell$-monoids. For the semilinear and idempotent case the nested sum representation helps us overcome this problem.

The next lemma shows that the nested sum decomposition of a finite subdirectly irreducible idempotent ordered monoid does not contain any consecutive occurrences of $\mathbf{C}_2$ or any consecutive occurrences of $\mathbf{C}_2^\partial$. 
\begin{lemma}\label{l:irred e-sum}
Every finite subdirectly irreducible idempotent ordered monoid is isomorphic to a nested sum $\bigboxplus^{n}_{i=1}\mathbf{M}_i$ with $\mathbf{M}_i \in \set{\mathbf{C}_2, \mathbf{C}_2^\partial, \mathbf{G}_3,\mathbf{D}_3}$ such that if  $\mathbf{M}_i = \mathbf{M}_{i+1}$, then $\mathbf{M}_i \in\set{ \mathbf{G}_3,\mathbf{D}_3}$ for every $i\in \set{1,\dots,n-1}$. 
\begin{proof}
Suppose that $\bf M$ is a finite idempotent ordered monoid. Then by \Cref{l: idemp ord monoid e-sum}, there is a unique nested sum decomposition $\mathbf{M} \cong \bigboxplus^{n}_{i=1}\mathbf{M}_i$ with  $\mathbf{M}_i \in \set{\mathbf{C}_2, \mathbf{C}_2^\partial, \mathbf{G}_3,\mathbf{D}_3}$. Suppose contrapositively that $\mathbf{M}_j = \mathbf{M}_{j+1}$ for some $j\in \set{1,\dots,n-1}$ with $\mathbf{M}_j \in \set{\mathbf{C}_2, \mathbf{C}_2^\partial}$. Let $a$ be the non-identity element  in  $\mathbf{M}_j$ and $b$ be the non-identity element in $\mathbf{M}_{j+1}$. Let $\mathbf{N}_1 =  \bigboxplus_{i=1}^{j-1}\mathbf{M}_i$, $\mathbf{N}_2 = \mathbf{M}_j\boxplus \mathbf{M}_{j+1}$, and $\mathbf{N}_3 = \bigboxplus_{k=j+2}^{n}\mathbf{M}_k$. 
Then $\mathbf{M} = \mathbf{N}_1\boxplus\mathbf{N}_2\boxplus\mathbf{N}_3$ and, by \Cref{thm:e-sum homomorphisms}, the homomorphism $\phi\colon \mathbf{N}_1\boxplus\mathbf{N}_2\boxplus\mathbf{N}_3 \to \mathbf{N}_1\boxplus\mathbf{M}_j\boxplus\mathbf{N}_3$ induced by the homomorphisms $\phi_1 = id_{\mathbf{N}_1}$, $\phi_2\colon \mathbf{N}_2 \to \mathbf{M}_j$ with $\phi_2(e) = e$, $\phi_2(a) =\phi_2(b) = a$, and $\phi_3 = id_{\mathbf{N}_3}$ has kernel $\ker(\phi) = \ker(\phi_1)\cup \ker(\phi_2) \cup \ker(\phi_3) = \Delta_{M} \cup\set{\pair{a,b},\pair{b,a}}$. Moreover, by \Cref{thm:e-sum homomorphisms}, the map $\psi \colon \mathbf{M} \to \bigboxplus^{n-1}_{i=1}\mathbf{M}_i$, induced by the homomorphisms $\psi_i = id_{\mathbf{M}_i}$ for $i=1,\dots,n-1$ and $\phi_n(c)= e$ for $c\in M_n$ is a homomorphism with kernel $\ker(\psi) = \Delta_{M} \cup M_n^2$. But then $\ker(\phi) \cap \ker(\psi) = \Delta_{ M}$, so $\bf M$ is not subdirectly irreducible.
\end{proof}
\end{lemma}

\begin{lemma}\label{l:e-sum subd irred lemma}
Let $\mathbf{M} = \bigboxplus^{n}_{i=1}\mathbf{M}_i$  with $\mathbf{M}_i \in \set{\mathbf{C}_2, \mathbf{C}_2^\partial, \mathbf{G}_3,\mathbf{D}_3}$ such that if  $\mathbf{M}_i = \mathbf{M}_{i+1}$, then $\mathbf{M}_i \in\set{ \mathbf{G}_3,\mathbf{D}_3}$ for every $i\in \set{1,\dots,n-1}$. Let $a,b\in M$ such that $a\neq b $, and $a\cdot b = a$ or $b \cdot a =a$. Then $\Theta(a,b) = \Theta(a,e)$. 
\begin{proof}
We show the claim for  $a\cdot b = a$ (the case  $b \cdot a =a$ is symmetric).
Since $a\cdot b = a$ and $e\cdot b=b $ we have $\Theta(a,b) \subseteq \Theta(a,e)$. Conversely, if $a\leq e\leq b$ or $b\leq e\leq a$, then we are done, since lattice congruence classes are convex. Otherwise using the assumptions that $a\neq b$ and $a\cdot b = a$, we get that either $e<b<a$ or $a<b<e$. Without loss of generality we assume that $e<b<a$.  So for $a\in M_i$ and $b\in M_j$ we get $i<j$, since both $a$ and $b$ are strictly positive. 

If $\mathbf{M}_i \in \set{ \mathbf{G}_3,\mathbf{D}_3}$ or $\mathbf{M}_j \in \set{ \mathbf{G}_3,\mathbf{D}_3}$, then there exists $c \in M_i \cup M_j$ with $c<e<b<a$ such that $ac = a$ and $bc = c$, or $ca = a$ and $cb = c$, yielding $\pair{a,c} \in \Theta(a,b)$ and, by convexity, $\pair{a,e}\in \Theta(a,b)$. 

Otherwise $\mathbf{M}_i = \mathbf{M}_j = \mathbf{C}_2^\partial$. So, by assumption, $i+1\neq j$ and $\mathbf{M}_{i+1} \neq \mathbf{C}_2^\partial$, i.e.,  $\mathbf{M}_{i+1} \in \set{\mathbf{C}_2, \mathbf{G}_3,\mathbf{D}_3}$. But then, there exists a  $c\in M_{i+1}$ such that $c<e<b<a$, $ac = a$, and $bc= c$, yielding $\pair{a,c} \in \Theta(a,b)$. Thus, by convexity, we get $\pair{a,e} \in \Theta(a,b)$.
\end{proof}
\end{lemma}

\begin{thm}\label{thm:subdirect-equivalence}
Let $\bf M$ be a non-trivial finite idempotent distributive $\ell$-monoid. Then the following are equivalent:
\begin{enumerate}[label =\textup{(\arabic*)}]
\item $\bf M$ is subdirectly irreducible.
\item $\mathbf{M} \cong \bigboxplus^{n}_{i=1}\mathbf{M}_i$ for some $n\in \mathbb{N}\setminus\set{0}$ with $\mathbf{M}_i \in \set{\mathbf{C}_2, \mathbf{C}_2^\partial, \mathbf{G}_3,\mathbf{D}_3}$ such that if  $\mathbf{M}_i = \mathbf{M}_{i+1}$, then $\mathbf{M}_i \in\set{ \mathbf{G}_3,\mathbf{D}_3}$ for every $i\in \set{1,\dots,n-1}$.
\item $\Con(\mathbf{M})$ is a chain.
\end{enumerate}
\begin{proof}
The implication (1) $\Rightarrow$ (2) follows from \Cref{l:irred e-sum} and the implication (3) $\Rightarrow$ (1) is clear, since $\bf M$ is finite. 

For the implication (2) $\Rightarrow$ (3) assume that (2) holds  and consider two principal congruences $\Theta(a,b)$ and $\Theta(c,d)$ for $a,b,c,d \in M$. If $a=b$ or $c=d$, then we have $\Theta(a,b) = \Delta_{M}$ or $\Theta(c,d)= \Delta_{M}$, yielding $\Theta(a,b)\subseteq\Theta(c,d)$  or  $\Theta(c,d)\subseteq\Theta(a,b)$. So suppose that $a\neq b$ and $c\neq d$. Then, by \Cref{lemma:Idemp}, $a\cdot b \in \set{a,b}$ and $c \cdot d \in \set{c,d}$. Thus, by \Cref{l:e-sum subd irred lemma}, $\Theta(a,b) \in \set{\Theta(a,e),\Theta(b,e)}$ and $\Theta(c,d) \in \set{\Theta(c,e),\Theta(d,e)}$. 
Without loss of generality we may assume that $\Theta(a,b) = \Theta(a,e)$ and $\Theta(c,d) = \Theta(c,e)$. There are two cases, either $a\cdot c = a$ or $a\cdot c = c$. If $a\cdot c = a$, we get $\pair{a,c} = \pair{a\cdot c, e\cdot c} \in \Theta(a,e)$, so $\Theta(c,e) \subseteq \Theta(a,e)$.  If $a\cdot c = c$, we get $\pair{a,c} = \pair{a\cdot e,a\cdot c} \in \Theta(c,e)$, so $\Theta(a,e) \subseteq \Theta(c,e)$. Hence we have either $\Theta(a,b) \subseteq \Theta(c,d)$ or $\Theta(c,d) \subseteq \Theta(a,b)$. So $\Con(\mathbf{M})$ is a chain.
\end{proof}
\end{thm}

Arguing similarly as in \Cref{thm:counting idemp chains}, yields  the following counting result:

\begin{cor}\label{cor:counting subd}
The number $\mathbf{S}(n)$ of subdirectly irreducible idempotent ordered monoids with $n\in \mathbb{N}\setminus\set{0}$ elements (up to isomorphism) is recursively defined by $\mathbf{S}(1) = 1$, $\mathbf{S}(2) = 2$,  $\mathbf{S}(3) = 4$, and
\begin{equation*}
\mathbf{S}(n) = \mathbf{S}(n-1) + 2\mathbf{S}(n-2) + 2\mathbf{S}(n-3) \quad (n>3).
\end{equation*}
\end{cor}

It follows from \Cref{thm:subdirect-equivalence}, that if the algebras $\mathbf{C}_2$ and $\mathbf{C}_2^\partial$ do not occur in  the nested sum decomposition of a non-trivial finite idempotent ordered monoid,  then it is subdirectly irreducible. As the next example shows this is no longer true in the infinite case.

\begin{exmp}
Consider the chain  $\omega$ of natural numbers and set for $n\in \omega$,  $\mathbf{M}_n = \mathbf{G}_3$ and $\mathbf{M} = \bigboxplus_{n \in \omega} \mathbf{M}_n$. Then, if we define for $n\in \omega$ the congruence $\theta_n = \Delta_{M} \cup (\bigcup_{m\geq n} M_m)^2$, we get $\bigcap_{n\in \omega} \theta_n = \Delta_{M}$. Hence $\mathbf{M}$ is not subdirectly irreducible. Thus, there are infinite nested sums that are not subdirectly irreducible even though they do not have any occurrence of $\mathbf{C}_2$ or $\mathbf{C}_2^\partial$ in their decomposition. 
\end{exmp}

The next theorem is the correct generalization of \Cref{thm:subdirect-equivalence}. 
\begin{thm}\label{thm:general subd equivalence}
Let  $\mathbf{M} = \bigboxplus_{c\in C} \mathbf{M}_c$ be an  idempotent ordered monoid such that for all $c\in C$, $\mathbf{M}_c  \in \set{\mathbf{C}_2, \mathbf{C}_2^\partial, \mathbf{G}_3,\mathbf{D}_3}$. Then the following are equivalent:
\begin{enumerate}[label=\textup{(\arabic*)}]
\item $\bf M$ is subdirectly irreducible.
\item $\mathbf{C}$ contains a maximal element and for all $d_1,d_2 \in C$ with $d_1\sqsubset d_2$ and $\mathbf{M}_{d_1}= \mathbf{M}_{d_2} \in \set{\mathbf{C}_2, \mathbf{C}_2^\partial}$ there exists  $b\in C$ such that $d_1\sqsubset  b \sqsubset d_2$ and $\mathbf{M}_b \neq \mathbf{M}_{d_1} $.
\item  $\Con(\mathbf{M})$ is a chain containing an atom.
\end{enumerate}
\end{thm}

\begin{proof}
The implication (3) $\Rightarrow$ (1) is clear.

(1) $\Rightarrow$ (2): Suppose that $\bf M$ is subdirectly irreducible. If $\bf C$ does not contain a maximal element, then we can define for every $c\in C$ the congruence $\theta_c = \Delta_M \cup (\bigcup_{c\sqsubseteq d} M_d)^2$ such that $\Delta_{ M} = \bigcap_{c\in C} \theta_c$, contradicting the fact that~$\bf M$ is subdirectly irreducible.
So $\bf C$ must contain a maximal element $m$ and we can define the congruence $\mu =  M_m^2 \cup \Delta_{M}$. Now suppose for a contradiction that there exist $d_1,d_2\in C$ with $d_1\sqsubset d_2$ such that for all $b\in [d_1,d_2]$, $\mathbf{M}_b = \mathbf{C}_2$. Then for $M_b\setminus\set{e} = \set{\bot_b}$ the relation  $\theta = \Delta_{ M} \cup \set{\pair{\bot_{b_1},\bot_{b_2}} \mid b_1,b_2 \in [d_1,d_2]}$ is a congruence of $\bf M$ and clearly $\mu \cap \theta = \Delta_{ M}$, a contradiction.  Dually we also get a contradiction if for all $b\in [d_1,d_2]$, $\mathbf{M}_b = \mathbf{C}_2^\partial$. Hence, for all $d_1,d_2 \in C$ with $d_1\sqsubset d_2$ and $\mathbf{M}_{d_1}= \mathbf{M}_{d_2} \in \set{\mathbf{C}_2, \mathbf{C}_2^\partial}$ there exists  $b\in C$ such that $d_1\sqsubset  b \sqsubset d_2$ and $\mathbf{M}_b \neq \mathbf{M}_{d_1} $.

(2) $\Rightarrow$ (3): Suppose that $\bf C$ contains a maximal element $m$ and for all $d_1,d_2 \in C$ with $d_1\sqsubset d_2$ and $\mathbf{M}_{d_2} = \mathbf{M}_{d_1} \in \set{\mathbf{C}_2, \mathbf{C}_2^\partial}$ there exists  $b\in C$ such that $d_1\sqsubset  b \sqsubset d_2$ and $\mathbf{M}_b \neq \mathbf{M}_{d_1} $. First note that $\mu = M_m^2 \cup \Delta_{M}$ is an atom in $\Con(\mathbf{M})$, since $\mathbf{M}_m$ is simple and $m$ is the maximal element of $\bf C$. Now let $\Theta^{\bf M}(a_1,b_1)$ and  $\Theta^{\bf M}(a_2,b_2)$ be  principal congruences for $a_1,a_2,b_1,b_2 \in M$.  Then, by assumption, there exists a \emph{finite} subchain $\bf D$ of $\bf C$ and nested sum $\mathbf{N} =  \bigboxplus_{d\in D} \mathbf{M}_d$  with $a_1,a_2,b_1,b_2 \in N$ such that $\bf N$ does not  contain consecutive occurrences of $\mathbf{C}_2$ or $\mathbf{C}_2^\partial$, respectively. But then, by \Cref{thm:subdirect-equivalence}, $\Con(\mathbf{N})$ is a chain and  we have either $\Theta(a_1,b_1)^{\bf M} \cap N^2 \subseteq \Theta^{\bf M}(a_2,b_2)\cap N^2$ or $\Theta^{\bf M}(a_2,b_2) \cap N^2 \subseteq \Theta^{\bf M}(a_1,b_1)\cap N^2$. Hence $\Theta^{\bf M}(a_1,b_1)  \subseteq \Theta^{\bf M}(a_2,b_2)$ or $\Theta^{\bf M}(a_2,b_2)  \subseteq \Theta^{\bf M}(a_1,b_1)$. Thus $\Con(\mathbf{M})$ is a chain containing an atom.
\end{proof}

\begin{cor}
Up to isomorphism the algebras $\mathbf{C}_2$, $\mathbf{C}_2^\partial$, $\mathbf{G}_3$, and $\mathbf{D}_3$ are exactly the simple semilinear idempotent distributive $\ell$-monoids.
\end{cor}

\section{The lattice of subvarieties of {\sf SemIdDLM}}\label{sec5}
In this section we use the nested sum representation and the characterization of the finite subdirectly irreducibles to investigate the lattice of subvarieties of $\Sem\Id$.

Recall that a poset $\alg{P,\leq}$ is called \emph{well-founded} if it neither contains an infinite antichain nor an infinite descending chain. For two algebras $\bf A$ and~$\bf B$ we define the relation $\leq_{HS}$ ($\leq_{IS}$) by $\mathbf{A} \leq_{HS} \mathbf{B}$ if and only if $\mathbf{A}\in HS(\mathbf{B})$ ($\mathbf{A} \leq_{IS} \mathbf{B}$ if and only if $\mathbf{A}\in IS(\mathbf{B})$). Moreover, for a variety $\var{V}$ of finite type we denote by $\var{V}_\ast$ a set which contains exactly one algebra from each of the isomorphism classes of the finite subdirectly irreducible members of $\var{V}$. Then $\alg{\var{V}_\ast,\leq_{HS}}$  and $\alg{\var{V}_\ast,\leq_{IS}}$ are posets. The following result by Davey shows how to describe the lattice of subvarieties of $\var{V}$ using $\alg{\var{V}_\ast,\leq_{HS}}$ if $\var{V}$ is congruence-distributive and locally finite, where we recall that a complete lattice $\mathbf{D} = \alg{D,\meet,\join}$ is called \emph{completely distributive} if for all index sets $I, J$ and $\set{a_{i,j}}_{i\in I,j\in J} \subseteq D$ we have 
\begin{equation*}
\bigmeet\left\{ \bigjoin \set{a_{i,j} \mid j\in J } \mid i \in I\right\} = \bigjoin \left\{\bigmeet\set{a_{i,f(i)} \mid i\in I} \mid f\colon I \to J\right\}.
\end{equation*}

\begin{thm}[{\cite[Theorem 3.3]{Davey1979}}]\label{thm:subvariety lattice}
Let $\var{V}$ be a congruence-distributive, locally finite variety of finite type. Then the lattice of subvarieties of $\var{V}$  is completely distributive and is isomorphic to the lattice of order ideals of the poset $\pair{\var{V}_\ast,\leq_{HS}}$ via the map  $\mathcal{I} \mapsto V(\mathcal{I})$ mapping an order ideal of $\pair{\var{V}_\ast,\leq_{HS}}$ to the subvariety of $\var{V}$ that it generates.
\end{thm}

\begin{cor}[{\cite[Corollary 4.3]{Olson2012}}]\label{l:well quasi-order finite subvariety lattice}
Let $\var{V}$ be a congruence-distributive, locally finite variety of finite type. Then the lattice of subvarieties of $\var{V}$ is countable if and only if $\pair{\var{V}_\ast,\leq_{HS}}$ is well-founded.
\end{cor}

As $\Sem\Id$ is a congruence-distributive, locally finite variety of finite type we want to use \Cref{thm:subvariety lattice} and \Cref{l:well quasi-order finite subvariety lattice}. By \Cref{l: idemp ord monoid e-sum}, we can assume that $\Sem\Id_\ast$ consists of the unique nested sum decompositions of the finite subdirectly irreducible algebras.

\begin{lemma}\label{l:IS=HS}
Every homomorphic image of an idempotent ordered monoid $\bf M$ is isomorphic to a subalgebra of $\bf M$.
\end{lemma}
\begin{proof}
We prove that for every congruence $\theta$ on $\bf M$, the quotient $\mathbf{M}/\theta$ is isomorphic to a subalgebra of $\bf M$.
For a congruence $\theta$ on $\bf M$ let $\phi\colon M \to M$  be a map  such that $\phi(e) =e$ and that maps every element to a fixed representative of its congruence class.  So in particular for the congruence class of $e$ the fixed representative is $e$. Since the congruence classes of $\theta$ are convex it is clear that $\phi$ is order-preserving. Let $a,b\in M$ with $a b = a$. If $cd = c$ for all $c,d \in M$ with $\pair{a,c},\pair{b,d}\in \theta$, then $\phi(a)\cdot \phi(b) = \phi(a)$. If $cd = d$ for some $c,d \in M$ with $\pair{a,c},\pair{b,d}\in \theta$,  then $\pair{a,d} = \pair{ab,cd} \in \theta$, yielding $\pair{a,b}\in \theta$. But then $\phi(a) = \phi(b)$, so $\phi(ab) = \phi(a) =  \phi(a)\phi(a) = \phi(a)\phi(b)$. Hence $\phi$ is a homomorphism and, by the definition of $\phi$ we have $\ker(\phi) = \theta$. Thus the claim follows by the homomorphism theorem.
\end{proof}

\begin{cor}\label{cor:HS=IS}
We have $\alg{\Sem\Id_\ast, \leq_{HS}} = \alg{\Sem\Id_\ast, \leq_{IS}}$.
\end{cor}

Let $\pair{P,\leq}$ be a poset. We define the order $\leq^\ast$ on the set $\sigma(P)$ of finite sequences of  $P$ by 
\begin{align*}
\pair{p_1,\dots,p_n} \leq^\ast \pair{q_1,\dots,q_m} &:\!\iff \text{there exists an order-embedding } \\
&\quad\quad\ \ \  f\colon \set{1,\dots,n} \to \set{1,\dots,m}   \text{ such that } \\
&\quad\quad\ \ \ p_i \leq q_{f(i)} \text{ for all } i\in \set{1,\dots,n}.
\end{align*}
\begin{lemma}[Higman's Lemma {\cite[Theorem 4.3]{Higman1952}}]\label{l:higman}
If $\pair{P,\leq}$ is a well-founded  poset, then $\pair{\sigma(P),\leq^\ast}$ is a well-founded poset.
\end{lemma}

\Cref{fig:subalgebra quasi-order} depicts the  order $\leq_{IS}$ restricted to  the set $\set{\mathbf{C}_2, \mathbf{C}_2^\partial, \mathbf{G}_3,\mathbf{D}_3}$, which is clearly well-founded.  The algebras $\mathbf{C}_2$ and $\mathbf{C}_2^\partial$ are incomparable with respect to $\leq_{IS}$, since they are not isomorphic, and the same holds for $\mathbf{G}_3$ and~$\mathbf{D}_3$. Moreover, the algebras $\mathbf{C}_2$ and $\mathbf{C}_2^\partial$ are subalgebras of the algebras $\mathbf{G}_3$ and $\mathbf{D}_3$.
So we have  $\mathbf{C}_2 \leq_{IS}\mathbf{G}_3$,  $\mathbf{C}_2 \leq_{IS}\mathbf{D}_3$, $\mathbf{C}_2^\partial \leq_{IS}\mathbf{G}_3$,  and $\mathbf{C}_2^\partial \leq_{IS}\mathbf{D}_3$.

\begin{figure}
\centering
\includegraphics{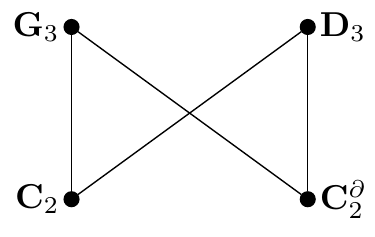}
\caption{The order $\leq_{IS}$ on the set $\set{\mathbf{C}_2, \mathbf{C}_2^\partial, \mathbf{G}_3,\mathbf{D}_3}$} 
\label{fig:subalgebra quasi-order}
\end{figure} 
Now if we consider the poset $\alg{\sigma(\set{\mathbf{C}_2, \mathbf{C}_2^\partial, \mathbf{G}_3,\mathbf{D}_3}),\leq_{IS}^\ast}$, by \Cref{cor:IS embedding}, it is immediate that for all  $\mathbf{M},\mathbf{N} \in \Sem\Id_\ast$ such that  $\mathbf{M} = \mathbf{M}_1\boxplus \cdots \boxplus \mathbf{M}_m $, $\mathbf{N} = \mathbf{N}_1\boxplus \cdots \boxplus \mathbf{N}_n$  with $ \mathbf{M}_i, \mathbf{N}_j \in \set{\mathbf{C}_2, \mathbf{C}_2^\partial, \mathbf{G}_3,\mathbf{D}_3}$ we have
\begin{equation}\label{eq:IS-higman}
\mathbf{M} \leq_{IS} \mathbf{N} \iff  \pair{\mathbf{M}_1,\dots,\mathbf{M}_m} \leq_{IS}^\ast \pair{\mathbf{N}_1,\dots,\mathbf{N}_n}.
\end{equation}
So, since restrictions of well-founded partial orders are well-founded,  \Cref{l:higman} yields that $\alg{\Sem\Id_\ast, \leq_{IS}}$ is well-founded. 
\begin{thm}
The lattice of subvarieties of $\Sem\Id$ is countably infinite.
\end{thm}
\begin{proof}
Since $\alg{\Sem\Id_\ast, \leq_{IS}}$ is well-founded, it follows from  \Cref{cor:HS=IS} and \Cref{l:well quasi-order finite subvariety lattice} that the lattice of subvarieties of $\Sem\Id$ is countable. That it is infinite follows from \Cref{l:jonsson} and the fact that, by \Cref{cor:counting subd},  $\Sem\Id$ contains infinitely many finite subdirectly irreducibles.
\end{proof}

\begin{remark}
By  \Cref{thm:subvariety lattice}, the lattice of subvarieties of $\Sem\Id$ is isomorphic to the  lattice of order ideals of $\alg{\Sem\Id_\ast, \leq_{IS}}$ via the map $\mathcal{I} \mapsto V(\mathcal{I})$. Moreover, by \Cref{thm:subdirect-equivalence}, we have a description of the finite subdirectly irreducible elements of $\Sem\Id$ in terms of the nested sum decomposition and, by the equivalence (\ref{eq:IS-higman}), we have a description of the order $\leq_{IS}$ on $\Sem\Id_\ast$ in terms of the nested sum decomposition. Hence, we get a more or less explicit description of the lattice of subvarieties of $\Sem\Id$ in terms of the nested sum decomposition of its finite subdirectly irreducible members.
\end{remark}

\section{The commutative case} \label{sec6}

In this section we apply the results of the previous two sections to the commutative case to get an explicit description of the lattice of subvarieties of $\Com\Id$. Using this description we obtain a finite axiomatization for every subvariety of $\Com\Id$.\footnote{The axiomatization was suggested to the author by Nick Galatos.}  

From the nested sum decomposition and the characterization of subdirectly irreducibles we immediately get the following result for the commutative case:
\begin{cor}\label{thm:e-sums commutative case}
Let $\mathbf{M}$ be a finite commutative idempotent ordered monoid. Then $\mathbf{M} \cong \bigboxplus^{n}_{i=1}\mathbf{M}_i$ with $\mathbf{M}_i \in \set{\mathbf{C}_2,\mathbf{C}_2^\partial}$. Moreover, $\bf M$ is subdirectly irreducible if and only if  for all $i\in \set{1,\dots,n-1}$ we have  $\mathbf{M}_i \neq \mathbf{M}_{i+1}$.
\end{cor}

Furthermore, we get the following counting result:

\begin{cor}\label{c:count commutative}
There are up to isomorphism $2^{n-1}$  commutative idempotent ordered monoids of size $n\geq 1$. 
\begin{proof}
Note that every commutative idempotent ordered monoid of size $n\geq 1$ has a unique nested sum representation $\bigboxplus^{n-1}_{i=1}\mathbf{M}_i$ with $\mathbf{M}_i \in \set{\mathbf{C}_2,\mathbf{C}_2^\partial}$. So it is clear that there are $2^{n-1}$ idempotent ordered monoids of size $n\geq 1$. 
\end{proof}
\end{cor}

\begin{remark}
Again \Cref{c:count commutative} also follows from the corresponding counting result about commutative idempotent residuated chains in \cite{Gil-FerezJipsenMetcalfe2020}. In \cite{Devillet2020}  a counting result is obtained for finite totally ordered commutative idempotent semigroups. 
\end{remark}

For $n>2$ we define inductively the algebras $\mathbf{C}_n$ and $\mathbf{C}_n^\partial$ by
\begin{align*}
\mathbf{C}_n &:= \mathbf{C}_2 \boxplus \mathbf{C}^\partial_{n-1} \\
\mathbf{C}_n^\partial &:= \mathbf{C}_2^\partial \boxplus \mathbf{C}_{n-1}
\end{align*}
and we set $\mathbf{C}_1 = \mathbf{C}_1^\partial = \mathbf{0}$. Note that these are exactly the commutative idempotent ordered monoids that do not contain consecutive occurrences of~$\mathbf{C}_2$ or $\mathbf{C}_2^\partial$.
Hence, \Cref{thm:e-sums commutative case} yields an explicit characterization of the subdirectly irreducible idempotent ordered monoids in the commutative case:
\begin{cor}
 For every $n>1$  the algebras $\mathbf{C}_n$ and $\mathbf{C}_n^\partial$ are up to isomorphism the only subdirectly irreducible commutative idempotent ordered monoids   with $n$ elements.
\end{cor}

\begin{remark}
For $n\geq 2$ the algebras $\mathbf{C}_n$ are exactly the distributive $\ell$-monoid reducts of the finite Sugihara chains of length $n$ (for the definition of a Sugihara monoid see e.g., \cite{Raftery2007}). Indeed, the same is true for $n \geq 2$ and the unique extension of $\mathbf{C}_n$ with an implication.
\end{remark}

If we let $\Com\Id_\ast = \set{\mathbf{C}_n,\mathbf{C}_n^\partial \mid n\geq 2}$ and we consider $\alg{\Com\Id_\ast, \leq_{IS}}$, then it follows from \Cref{cor:IS embedding} that for $2 \leq m<n$ we have $\mathbf{C}_m \leq_{IS} \mathbf{C}_n$, $\mathbf{C}_m^\partial \leq_{IS} \mathbf{C}_n^\partial$,  $\mathbf{C}_m \leq_{IS} \mathbf{C}_n^\partial$, and  $\mathbf{C}_m^\partial \leq_{IS} \mathbf{C}_n$.  Thus  the poset $\alg{\Com\Id_\ast, \leq_{IS}}$ has the form represented by \Cref{fig:IS-order on CId}.

\begin{figure}
\centering 
\includegraphics{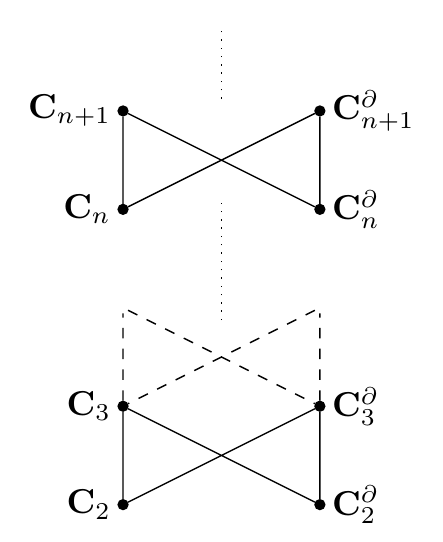}
\caption{The poset $\alg{\Com\Id_\ast, \leq_{IS}}$}
\label{fig:IS-order on CId}
\end{figure}

If for $X\subseteq \Com\Id_\ast$ we denote by ${\downarrow}  X$ the downwards closure of $X$ under $\leq_{IS}$, then it is clear from \Cref{fig:IS-order on CId} that every proper order ideal of  $\alg{\Com\Id_\ast, \leq_{IS}}$ is equal to either ${\downarrow}\set{\mathbf{C}_n}$, ${\downarrow}\set{\mathbf{C}_n^\partial}$ or ${\downarrow}\set{\mathbf{C}_n,\mathbf{C}_n^\partial}$ for some $n\in \mathbb{N}$. Thus \Cref{cor:HS=IS} together with \Cref{thm:subvariety lattice} yields the following characterization of the lattice of subvarieties of $\Com\Id$:

\begin{figure}
\centering 
\includegraphics{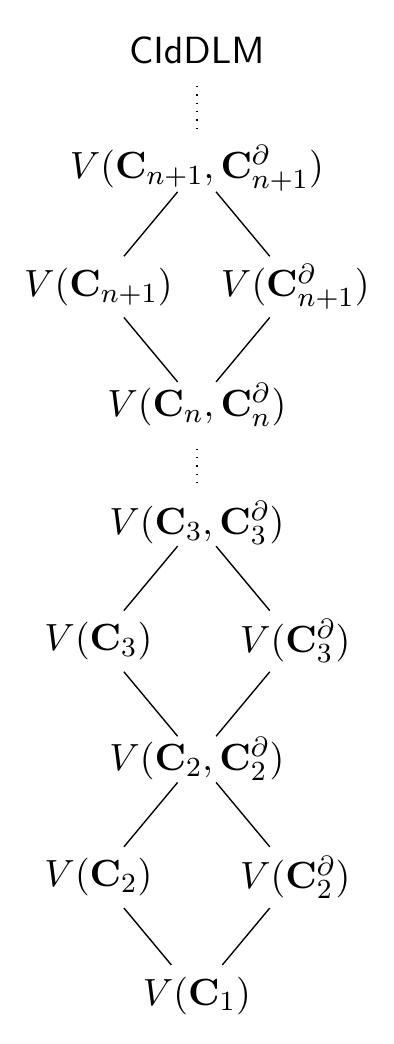}
\caption{The lattice of subvarieties of  $\Com\Id$}
\label{fig:Subvariety lattice of CId}
\end{figure}

\begin{thm}\label{thm:CId lattice}
The lattice of subvarieties of $\Com\Id$ is the lattice represented by \textnormal{\Cref{fig:Subvariety lattice of CId}}.
\end{thm}

Let us recall that a class of algebras  $\var{K}$ is called a  \emph{prevariety} if it is closed under $I$, $S$, and $P$ or equivalently $\var{K} = ISP(\var{K})$.  
\begin{cor}\label{cor:subv-Com}
Every non-trivial proper subvariety of $\Com\Id$ can be generated as a prevariety by two or fewer finite subdirectly irreducible algebras.
\end{cor}

\begin{cor}
Every infinite subdirectly irreducible member of $\Com\Id$ generates $\Com\Id$ as a quasi-variety.
\end{cor}

Let us for every natural number  $n\geq 2$ rename the elements of $\mathbf{C}_n$ such that $\mathbf{C}_n = \alg{\set{e,1,\dots, n-1}, \cdot, e, \leq}$ with
\begin{equation}\label{eq:Cn}
C_n =  \begin{cases}
\set{n-1 <  \dots < 3 < 1 < e < 2 <  4 < \dots < n-2 } & n \text{ even}, \\
\set{n-1 <  \dots < 4 < 2 < e < 1 <  3 < \dots < n-2 } & n \text{ odd},
\end{cases}
\end{equation}
 and $m \cdot k = \max_{\mathbb{N}}(m,k)$, and dually $\mathbf{C}_n^\partial = \alg{\set{e,1,\dots, n-1}, \cdot, e,\leq}$ with
\begin{equation}\label{eq:Cnd}
C_n^\partial =  \begin{cases}
\set{n-2 < \dots < 4 < 2 < e < 1 <  3 < \dots < n-1 } & n \text{ even}, \\
\set{n-2  < \dots < 3 < 1 < e < 2 <  4 < \dots < n-1 } & n \text{ odd},
\end{cases}
\end{equation}
and $m \cdot k = \max_{\mathbb{N}}(m,k)$. We define the equations $\sigma_n$ for $n\geq 2$ inductively as follows: we set $\sigma_2  \coloneqq ( x_1 \leq e) $ and if $\sigma_n = (s_n \leq t_n)$ is  already defined, we set
\begin{equation*}
\sigma_{n+1} \coloneqq \begin{cases}
s_n \meet x_{n-1}x_n \leq t_n & \text{if $n+1$ is even}, \\
s_n \leq t_n \join x_{n-1}x_n & \text{if $n+1$ is odd}.
\end{cases}
\end{equation*}
So for $n \geq 3$  we have
\begin{equation*}
\sigma_n = \begin{cases}
x_1 \meet x_2x_3 \meet  \dots \meet x_{n-2}x_{n-1} \leq e \join x_1x_2 \join \dots \join x_{n-3}x_{n-2} & \text{$n$ even},\\
x_1 \meet x_2x_3 \meet  \dots \meet x_{n-3}x_{n-2} \leq e \join x_1x_2 \join \dots \join x_{n-2}x_{n-1} & \text{$n$  odd}.
\end{cases}
\end{equation*}

\begin{lemma}\label{l:axiom}
\hfill
\begin{enumerate}[label =\textup{(\roman*)}]
\item Let  $n\geq 2$ be even. Then we have $\mathbf{C}_n^\partial \not\models \sigma_n$. Moreover,  for every $\mathbf{M} \in \Com\IdOM$  with $\mathbf{M}\not\models \sigma_n$ and elements  $a_1,\dots,a_{n-1} \in M$ witnessing the failure of $\sigma_n$ via the valuation $x_k \mapsto a_k$, the subalgebra $\Sg(a_1,\dots,a_{n-1})$ is isomorphic to $\mathbf{C}_n^\partial$ via the map $a_k \mapsto k$.  
\item   Let  $n\geq 2$ be odd.  Then we have  $\mathbf{C}_n \not\models \sigma_n$.  Moreover, for every $\mathbf{M} \in \Com\IdOM$ with $\mathbf{M}\not\models \sigma_n$ and elements $a_1,\dots,a_{n-1} \in M$ witnessing the failure of $\sigma_n$ via the valuation $x_k \mapsto a_k$, the subalgebra $\Sg(a_1,\dots,a_{n-1})$ is isomorphic to $\mathbf{C}_n$ via the map $a_k \mapsto k$.
\end{enumerate}
\end{lemma}

\begin{proof}
For the first part of (i) note that for $n=2$ we clearly have $1 > e$ in~$\mathbf{C}_2^\partial$, and for $n> 2$ even  and $1,\dots, n-1 \in {C}_n^\partial$  we get, by \Cref{eq:Cnd}, that
\begin{align*}
 1 \meet 2\cdot 3 \meet  \dots \meet (n-2)\cdot (n-1) &= 1  \\
 &> e \\
&= e \join 1\cdot 2 \join \dots \join (n-3)\cdot (n-2).
\end{align*}
Thus  $\mathbf{C}_n^\partial \not\models \sigma_n$. 
Similarly we get for the first part of (ii) that for $n\geq 2$ odd, $\mathbf{C}_n \not\models \sigma_n$.

Now we prove the second parts of (i) and (ii) together by induction on $n\geq 2$. For $n=2$ we have $\sigma_2 = (x_1 \leq e)$ and  if $\mathbf{M} \in \Com\IdOM$ and $a_1 \in M$ such that $a_1\nleq e$, then, since $\mathbf{M}$ is totally ordered, we get $a_1>e$ and the subalgebra $\Sg(a_1) = \alg{\set{e,a_1}, \leq, \cdot,e}$ is clearly isomorphic to $\mathbf{C}_2^\partial$.

Now suppose that $n>2$ is odd and  (i) holds for $n-1$. Then for $\sigma_{n-1} = (s_{n-1} \leq t_{n-1})$,  $\sigma_n$ has the following form
\begin{equation*}
s_{n-1} \leq t_{n-1} \join x_{n-2}x_{n-1}
\end{equation*}
Let $\mathbf{M}\in \Com\IdOM$ and $a_1,\dots, a_{n-1} \in M$ witnessing the failure of $\sigma_n$. Then, since $\mathbf{M}$ is totally ordered we have
\begin{equation*}
s_{n-1}^{\bf M}(a_1,\dots,a_{n-1})  > t_{n-1}^{\bf M}(a_1,\dots, a_{n-1}) \join a_{n-2}a_{n-1}
\end{equation*}
Hence we get
\begin{equation*}
s_{n-1}^{\bf M}(a_1,\dots,a_{n-1})  > t_{n-1}^{\bf M}(a_1,\dots, a_{n-1})
\end{equation*}
So the elements $a_1,\dots,a_{n-2}$ witness the failure of $\sigma_{n-1}$ via the map $x_k \mapsto a_k$ and the induction hypothesis yields that $\Sg(a_1,\dots, a_{n-2})$  is isomorphic to~$\mathbf{C}_{n-1}^\partial$ via the map $a_k \mapsto k$. Thus, by \Cref{eq:Cnd}, we get  that 
\begin{equation*}
\sg(a_1,\dots, a_{n-2}) =  \set{a_{n-3} < \dots a_{4} < a_{2} < e < a_1 < a_3 < \dots < a_{n-2}}
\end{equation*}
and $a_i\cdot a_j = a_{\max(i,j})$ for $i,j \in \set{1,\dots,n-1}$. Now since $s_{n-1}$ is a meet containing $x_1$ we get
\begin{equation*}
a_{n-2} \geq a_1 > a_{n-2}a_{n-1},
\end{equation*}
yielding $e > a_{n-1}$ and $a_{n-2}a_{n-1} = a_{n-1}$. Thus, since $a_{n-2} \cdot a_{n-3} = a_{n-2}$, we get $a_{n-1} < a_{n-3}$ and, by \Cref{eq:Cn},  $\Sg(a_1,\dots, a_{n-1})$ is isomorphic to $\mathbf{C}_{n}$ via the map $a_k \mapsto k$. 

The case where $n$ is even and (ii) holds for $n-1$ is very similar.
\end{proof}
For $n\geq 2$ and two disjoint sets $\set{x_1,\dots, x_{n-1}}$ and $\set{y_1,\dots,y_{n-1}}$ of distinct variables let $\sigma_n(x_1,\dots,x_{n-1}) = (s_n \leq t_n)$ and  $\sigma_n^\partial(y_1,\dots, y_{n-1})  = (t_n^\partial \leq s_n^\partial)$, and define $\gamma_n \coloneqq (s_n t_n^\partial  \leq t_n s_n^\partial)$. Note that $s_n^\partial$ and $t_n^\partial$ are strictly speaking not the duals of $s_n$ and $t_n$, since we change the variables.

\begin{thm}\label{thm:axiomatizations CId}
\hfill
\begin{enumerate}[label =\textup{(\arabic*)}]
\item If $n\geq 2$ is even, then the equations $\sigma_n$ and $\sigma_n^\partial$ axiomatize the subvarieties $V(\mathbf{C}_n)$ and $V(\mathbf{C}_n^\partial)$, respectively, relative to $\Com\Id$.
\item If $n\geq 2$ is odd, then the equations $\sigma_n$ and $\sigma_n^\partial$ axiomatize the subvarieties $V(\mathbf{C}_n^\partial)$ and $V(\mathbf{C}_n)$, respectively, relative to $\Com\Id$.
\item For $n\geq 2$ the equation $\gamma_{n+1}$ axiomatizes the subvariety $V(\mathbf{C}_n, \mathbf{C}_n^\partial)$ relative to $\Com\Id$.
\end{enumerate}
\end{thm}

\begin{proof}
First note that it follows from \Cref{thm:CId lattice} that if an equation holds in  $\mathbf{C}_n$ but not in $\mathbf{C}_n^\partial$, then it axiomatizes $V(\mathbf{C}_n)$ and, dually, if an equation holds in  $\mathbf{C}_n^\partial$ but not in $\mathbf{C}_n$, then it axiomatizes  $V(\mathbf{C}_n^\partial)$. Thus (1) and (2) follow immediately from \Cref{l:axiom} and the fact that $\mathbf{C}_n^\partial$  is the order dual of $\mathbf{C}_n$. For (3) note that  if two varieties $\var{V}_1$ and $\var{V}_2$ are axiomatized by sets of equations $\Sigma_1$ and $\Sigma_2$, respectively, then $\var{V}_1 \cap \var{V}_2$ is axiomatized by $\Sigma_1 \cup \Sigma_2$. Thus $V(\mathbf{C}_n, \mathbf{C}_n^\partial) = V(\mathbf{C}_{n+1}) \cap V(\mathbf{C}_{n+1}^\partial)$ is axiomatized by the set $\set{\sigma_{n+1}, \sigma_{n+1}^\partial}$. But also for every idempotent ordered monoid $\bf M$ and $a,b,c,d \in M$ we have that $a\leq b$ and $c\leq d$ implies $ac\leq bd$. So, by definition of $\gamma_{n+1}$, we get that  $V(\mathbf{C}_n, \mathbf{C}_n^\partial) \models \gamma_{n+1}$. On the other hand if we substitute in $\gamma_{n+1}$ the variables $y_1,\dots,y_n$ by $e$, we obtain $\sigma_{n+1}$ and if we substitute the variables $x_1,\dots,x_n$ by $e$,  we obtain $\sigma_{n+1}^\partial$. So $\gamma_{n+1}$ axiomatizes $V(\mathbf{C}_n, \mathbf{C}_n^\partial)$ relative to $\Com\Id$.
\end{proof}

\begin{remark}
It follows from \Cref{thm:CId lattice} and  the fact that for every $n\in \mathbb{N}\setminus\set{0}$ we have  $V(\mathbf{C}_n) = ISP(\mathbf{C}_n)$ that every finite member of  $\Com\Id$ embeds into a power of $\mathbf{C}_n$ for some  $n\in \mathbb{N}\setminus\set{0}$. Moreover, clearly every finite distributive $\ell$-monoid with an absorbing bottom element is the residual-free reduct of a fully distributive residuated lattice\footnote{A residuated lattice is called fully distributive if its residual-free reduct is a distributive $\ell$-monoid (see e.g., \cite{Cardona2017}).}. Thus the class of Sugihara monoids and the class of commutative idempotent fully distributive residuated lattices satisfy the same universal first-order sentences in the language $\set{\meet,\join,\cdot,e}$. This slightly generalizes Corollary 25 of  \cite{Raftery2007}. Furthermore it follows from \Cref{thm:axiomatizations CId} that non-isomorphic finite Sugihara chains can be distinguished by equations in the language $\set{\meet,\join,\cdot,e}$.
\end{remark}

\section{Amalgamation}\label{sec7}
In this section we use the nested sum representation to study the amalgamation property for idempotent ordered monoids. 

Let $\var{K}$ be a class of algebras of the same type. A \emph{span} in $\var{K}$ is a pair $\pair{i_1\colon \mathbf{A} \hookrightarrow \mathbf{B}, i_2 \colon \mathbf{A}\hookrightarrow \mathbf{C}}$ of embeddings between algebras $\mathbf{A},\mathbf{B},\mathbf{C} \in \var{K}$. The class $\var{K}$ has the \emph{amalgamation property} if for every span $\pair{i_1\colon \mathbf{A} \hookrightarrow \mathbf{B}, i_2 \colon \mathbf{A}\hookrightarrow \mathbf{C}}$   in~$\var{K}$ there exists a $\mathbf{D}\in \var{K}$ and embeddings $j_1\colon \mathbf{B} \hookrightarrow \mathbf{D}$, $j_2\colon \mathbf{C} \hookrightarrow \mathbf{D}$ such that $j_1\circ i_1 = j_2\circ i_2$, i.e., the following diagram commutes:
\[\begin{tikzcd}
	& {\bf B} \\
	{\bf A} && {\bf D} \\
	& {\bf C}
	\arrow["{i_1}", hook, from=2-1, to=1-2]
	\arrow["{j_1}", dashed, hook, from=1-2, to=2-3]
	\arrow["{i_2}"', hook, from=2-1, to=3-2]
	\arrow["{j_2}"', dashed, hook, from=3-2, to=2-3]
\end{tikzcd}\]
The triple $\pair{\mathbf{D},j_1,j_2}$ is called an \emph{amalgam} of the span. We say that $\var{K}$ has the \emph{strong amalgamation property} if for every span $\pair{i_1\colon \mathbf{A} \hookrightarrow \mathbf{B}, i_2 \colon \mathbf{A}\hookrightarrow \mathbf{C}}$ there exists an amalgam $\pair{\mathbf{D},j_1,j_2}$ such that  $j_1[B] \cap j_2[C] = (j_1\circ i_1)[A]$ which we call a \emph{strong amalgam} of the span.

\begin{prop}
The variety $\DLM$ does not have the amalgamation property.
\end{prop}

\begin{proof}
The set of invertible elements of a distributive $\ell$-monoid gives rise to a subalgebra which is the inverse-free reduct of an $\ell$-group. Suppose for a contradiction that $\DLM$ has the amalgamation property. Then every span of inverse-free reducts of $\ell$-groups has an amalgam and, by taking the subalgebra of invertible elements of the amalgam, we can assume that this amalgam is the inverse-free reduct of an $\ell$-group, since monoid homomorphisms map invertible elements to invertible elements.  Hence, it follows that the variety of $\ell$-groups has the amalgamation property, contradicting \cite[Theorem 3.1]{Pierce1972}.
\end{proof}

It is well-known that the variety of distributive lattices does not have the strong amalgamation property. We will now show that this failure of the strong amalgamation property is inherited by every non-trivial subvariety of $\DLM$.

\begin{lemma}\label{l:atoms}
The only two atoms in the lattice of subvarieties of $\DLM$ are $V(\mathbf{C}_2)$ and $V(\mathbf{C}_2^\partial)$. Moreover,  $V(\mathbf{C}_2)$ consists of all algebras $\alg{D,\meet,\join,\meet,e}$, where $\alg{D,\meet,\join}$ is a distributive lattice with greatest element $e$, and  $V(\mathbf{C}_2^\partial)$  consists of all algebras $\alg{D,\meet,\join,\join,e}$, where $\alg{D,\meet,\join}$ is a distributive lattice with least element $e$.
\end{lemma}
\begin{proof}
Let $\var{V}$ be a non-trivial subvariety of $\DLM$. Then there is an $\mathbf{M}\in \var{V}$ that is non-trivial, i.e., there is an $a_0\in M$ with $a_0\neq e$. But then either $a_0\meet e < e$ or $a_0\join e > e$. If $a = a_0\meet e < e$, then we consider the subalgebra $\bf A$ of~$\bf M$ generated by $a$. We have $A = \set{a^{n} \mid n\in \mathbb{N}}$ and  $a^{n+1} \leq a^n$, so the map $\phi\colon A \to C_2$ defined by $\phi(e) = e$ and $\phi(a^n) = \bot$ for $n>0$ is a homomorphism. Hence $\mathbf{C}_2 \in \var{V}$. Otherwise if $a_0 \join e > e$ we consider $a= a_0 \join e$ and, arguing dually, we get $\mathbf{C}_2^\partial \in \var{V}$.

The second part follows from the fact that $\mathbf{C}_2 = \alg{\set{\bot,e},\meet,\join,\meet,e}$, where $\alg{\set{\bot,e},\meet,\join}$ is the two element distributive lattice, and $V(\mathbf{C}_2) = ISP(\mathbf{C}_2)$. Dually the claim for $V(\mathbf{C}_2^\partial)$ follows.
\end{proof}

\begin{prop}
No non-trivial subvariety of $\DLM$ has the strong amalgamation property.
\end{prop}

\begin{proof}
Let $\var{V}$ be a non-trivial subvariety of $\DLM$.
By \Cref{l:atoms}, we may assume, without loss of generality, that $V(\mathbf{C}_2^\partial) \subseteq \var{V}$, otherwise $V(\mathbf{C}_2) \subseteq \var{V}$ and the argument is dual.  We consider the three-element chain $\mathbf{C} = \alg{\set{0,1,2},\meet,\join,\join,0}$ with $0<1<2$, and  the four-element distributive lattice   $\mathbf{D} =\alg{\set{\bot,a,b,\top},\meet,\join,\join,\bot}$  with $\bot<a<\top$, $\bot<b<\top$, and $a$ and $b$ incomparable.  Then, by \Cref{l:atoms}, $\mathbf{D}, \mathbf{C} \in \var{V}$. Define the maps $\phi \colon C \to D$ by $\phi(0) = \bot$, $\phi(1) = a$, $\phi(2) = \top$, and $\psi\colon C \to D$ by $\psi(0) =\bot$, $\psi(1) = b$, $\psi(2) = \top$. It is easy to see that $\phi$ and $\psi$ are $\ell$-monoid embeddings. Moreover, it is well-known (see e.g., \cite{Fried1990}) that the span $\pair{\phi,\psi}$ considered as a span in the variety of distributive lattices does not have a strong amalgam, so it does not have a strong amalgam in $\var{V}$,  since the lattice reduct of a strong amalgam in $\var{V}$ would be a strong amalgam in the variety of distributive lattices.
\end{proof}

Let  $\mathcal{S} = \pair{i_1\colon \mathbf{L} \hookrightarrow \mathbf{M}$, $i_2 \colon \mathbf{L}\hookrightarrow \mathbf{N}}$ and $\mathcal{S}' = \pair{j_1\colon \mathbf{L}'\hookrightarrow \mathbf{M}'$, $J_2 \colon \mathbf{L}'\hookrightarrow \mathbf{N}'}$  be  spans of idempotent ordered monoids. We say that $\mathcal{S}$ \emph{restricts to} $\mathcal{S}'$ if the algebras $\mathbf{L}'$, $\mathbf{M}'$, and $\mathbf{N}'$ are subalgebras of $\mathbf{L}$, $\mathbf{M}$, and $\mathbf{N}$, respectively, and ${i_1\!\!\restriction_{L'}} = j_1$ and ${i_2\!\!\restriction_{L'}} = j_2$. Moreover, we extend this notion in the obvious way to the case where the algebras $\mathbf{L}'$, $\mathbf{M}'$, and $\mathbf{N}'$ embed into $\mathbf{L}$, $\mathbf{M}$, and $\mathbf{N}$, respectively, by identifying $\mathbf{L}'$, $\mathbf{M}'$, and $\mathbf{N}'$ with the images of the  respective embeddings.

Recall from \Cref{sec5} that  $\mathbf{C}_2$  embeds into $\mathbf{G}_3$  and   $\mathbf{D}_3$ via the obvious inclusion maps. Moreover, clearly these are the only possible embeddings from $\mathbf{C}_2$ into   $\mathbf{G}_3$  and  $\mathbf{D}_3$, respectively. We denote the resulting span by $\mathcal{S}_1 := \pair{\mathbf{C}_2 \hookrightarrow \mathbf{G}_3,\mathbf{C}_2\hookrightarrow \mathbf{D}_3}$  and we define the span $\mathcal{S}_2 :=\pair{\mathbf{C}_2^\partial \hookrightarrow \mathbf{G}_3,\mathbf{C}_2^\partial \hookrightarrow \mathbf{D}_3}$ dually. From here on we will denote the top element of $\mathbf{G}_3$ by $\top$ and the top element of $\mathbf{D}_3$ by $\top^\ast$.
 We call a span  $\pair{i_1\colon \mathbf{L} \hookrightarrow \mathbf{M}$, $i_2 \colon \mathbf{L}\hookrightarrow \mathbf{N}}$ of idempotent ordered monoids a \emph{compatible span} if it does not restrict to any of the spans $\mathcal{S}_1$ or $\mathcal{S}_2$ up to permuting the embeddings in the span.

\begin{prop}\label{thm:compatible spans have strong amalgams}
Let $\mathbf{L}, \mathbf{M}, \mathbf{N}$ be idempotent ordered monoids. Then every compatible span $\pair{i_1\colon \mathbf{L} \hookrightarrow \mathbf{M}$, $i_2 \colon \mathbf{L}\hookrightarrow \mathbf{N}}$ has an amalgam in the class $\IdOM$ of idempotent ordered monoids. In particular, every span of  commutative idempotent ordered monoids has a strong amalgam.
\end{prop}
\begin{proof}
Let $\mathbf{A},\mathbf{B},\mathbf{C}$ be chains,  $\mathbf{L} = \bigboxplus_{a\in A} \mathbf{L}_a$, $\mathbf{M} = \bigboxplus_{b\in B} \mathbf{M}_b$, $\mathbf{N} = \bigboxplus_{c\in C} \mathbf{N}_c$ such that $\mathbf{L}_a,\mathbf{M}_b,\mathbf{N}_c\in \set{\mathbf{C}_2, \mathbf{C}_2^\partial, \mathbf{G}_3,\mathbf{D}_3}$, and
let $\pair{i_1\colon \mathbf{L} \to \mathbf{M}$, $i_2 \colon \mathbf{L}\to \mathbf{N}}$ be a compatible span. Then, by \Cref{l:embeddings to summands}, there are corresponding order-embeddings $f\colon \mathbf{A} \to \mathbf{B}$ and $g\colon \mathbf{A} \to \mathbf{C}$. Without loss of generality we may assume that $A = B\cap C$ and $f,g$ are the inclusion maps. By \cite[Section 9.2]{Fraisse1954}, there exists an extension $\bf D$ of the chains $\bf B$ and $\bf C$ with $D = B\cup C$. We define for every $a\in A$ the algebra $\mathbf{P}_a$ as the largest algebra of the set $\set{\mathbf{M}_a,\mathbf{N}_a}$. Note that the algebras $\mathbf{M}_a$ and $\mathbf{N}_a$ cannot be two \emph{different} algebras of the same cardinality, since clearly  $\set{\mathbf{M}_a,\mathbf{N}_a} = \set{\mathbf{C}_2,\mathbf{C}_2^\partial}$ is not possible because $\mathbf{L}_a$ is a non-trivial subalgebra of $\mathbf{M}_a$ and $\mathbf{N}_a$, and $\set{\mathbf{M}_a,\mathbf{N}_a} = \set{\mathbf{G}_3,\mathbf{D}_3}$  would contradict the assumption that the span is compatible. Hence $\mathbf{P}_a$ is well-defined for each $a\in A$ and we get, in particular, that both  $\mathbf{M}_a$ and $\mathbf{N}_a$ embed into  $\mathbf{P}_a$ and the embeddings agree on the images of $\mathbf{L}_a$ under $i_1$ and $i_2$.
Moreover, we define the algebras $\mathbf{P}_d$ for $d\in D\setminus A$ by
\begin{equation*}
\mathbf{P}_d =
\begin{cases}
\mathbf{M}_d & \text{if } d\in B\setminus A, \\
\mathbf{N}_d & \text{if } d\in C\setminus A.
\end{cases}
\end{equation*}
Then, by construction, for every $b\in B$, $\mathbf{M}_b$ embeds into $\mathbf{P}_b$ and for every $c\in C$, $\mathbf{N}_c$ embeds into $\mathbf{P}_c$. Define $\mathbf{P} =  \bigboxplus_{d\in D} \mathbf{P}_d$. Then, by construction,  \Cref{l:embeddings to summands} yields that the maps  $j_1 \colon \mathbf{M} \hookrightarrow \mathbf{P}$ and $j_2 \colon \mathbf{N} \hookrightarrow \mathbf{P}$ corresponding to the inclusions $\mathbf{B} \hookrightarrow \mathbf{D}$ and $\mathbf{C} \hookrightarrow \mathbf{D}$ are embeddings and together with the algebra $\bf P$ they form an amalgam for the span. 

If $\mathbf{L}$, $\mathbf{M}$, and $\mathbf{N}$ are commutative, then the algebras $\mathbf{G}_3$ and $\mathbf{D}_3$ cannot occur in their $e$-sum decomposition, i.e., the span is always compatible. Moreover, in the construction of $\mathbf{P}$ we have $\mathbf{P}_a  = \mathbf{N}_a = \mathbf{M}_a$ for every $a\in A$, so $\mathbf{P}$ is a strong amalgam.
\end{proof}

\begin{cor}\label{c:com amalgamation}
The class $\Com\IdOM$ of commutative idempotent ordered monoids has the strong amalgamation property.
\end{cor}
The next example shows that the amalgam constructed in the proof of \Cref{thm:compatible spans have strong amalgams} is not necessarily a strong amalgam and even if we restrict to compatible spans a strong amalgam does not always exist. 
\begin{exmp}
Consider the span $\pair{ \mathbf{C}_2 \hookrightarrow \mathbf{G}_3,\mathbf{C}_2\hookrightarrow \mathbf{G}_3}$. Clearly it is compatible, but it does not have a strong amalgam in $\IdOM$. To see this let $\pair{\mathbf{M}, j_1 \colon \mathbf{G}_3 \to \mathbf{M}, j_2 \colon \mathbf{G}_3 \to \mathbf{M}}$ be an amalgam of the span in $\IdOM$. Then, by \Cref{l:embeddings to summands},   $j_1(\mathbf{G}_3)$ and $j_2(\mathbf{G}_3)$ are components of the nested sum decomposition of $\mathbf{M}$ isomorphic to $\mathbf{G}_3$. But we have $j_1(\bot) \in  j_1(G_3) \cap  j_2(G_3)$, so we get $j_1(G_3) =  j_2(G_3)$, i.e., the amalgam is not strong.
\end{exmp}
Note that every span $\pair{i_1\colon \boldsymbol{0} \hookrightarrow \mathbf{L}, i_2 \colon \boldsymbol{0} \hookrightarrow \mathbf{M}}$ is compatible. So we get the following result.

\begin{cor}
The class $\IdOM$ of idempotent ordered monoids has the joint embedding property, i.e., each pair $\bf L$, $\bf M$ of idempotent ordered monoids embeds into a common idempotent ordered monoid  $\bf N$.
\end{cor}

The next proposition shows that being compatible is not only a sufficient condition for spans of idempotent ordered monoids to have an amalgam in the class of ordered monoids but also a necessary condition. 

\begin{prop}\label{prop:not compatible no amalgam}
The spans $\mathcal{S}_1$ and $\mathcal{S}_2$ do not have an amalgam in the class $\OM$ of ordered monoids.
\end{prop}
\begin{proof}
Suppose for a contradiction that there exists an amalgam $\bf E$ in the class of ordered monoids for the span $\mathcal{S}_1
$ with embeddings $j_1 \colon \mathbf{G}_3\hookrightarrow \mathbf{E}$ and $j_2 \colon \mathbf{D}_3 \hookrightarrow \mathbf{E}$, then,  either $j_1(\top) \leq j_2(\top^\ast)$ or $j_2(\top^\ast) \leq j_1(\top)$. If we have $j_1(\top) \leq j_2(\top^\ast)$, then we get
\begin{align*}
j_1(\top) = j_1(\top\cdot \bot) &=  j_1(\top)\cdot j_1(\bot) \\
&=  j_1(\top)\cdot j_2(\bot) \\
& \leq  j_2(\top^\ast)\cdot j_2(\bot) \\
&= j_2(\top^\ast \cdot \bot) \\
&= j_2(\bot) \\
 &= j_1(\bot).
\end{align*}
So, since $j_1$ is order-preserving, $j_1(\top) = j_1(\bot)$, contradicting the injectivity of $j_1$. Similarly $j_2(\top^\ast) \leq j_1(\top)$, yields $j_2(\top^\ast) = j_2(\bot)$ which is again a contradiction. Hence $\mathcal{S}_1$ does not have an amalgam in $\OM$. For the span $\mathcal{S}_2$ the claim follows by duality.
\end{proof}
\begin{cor}
The classes $\OM$ and $\IdOM$ do not have the amalgamation property.
\end{cor}
From \Cref{prop:not compatible no amalgam} and \Cref{thm:compatible spans have strong amalgams} we get the following result.
\begin{thm}\label{thm:amalgamation of chains}
For a span $\mathcal{S} = \pair{i_1\colon \mathbf{L} \hookrightarrow \mathbf{M}$, $i_2 \colon \mathbf{L}\hookrightarrow \mathbf{N}}$ of  idempotent ordered monoids the following are equivalent:
\begin{enumerate}[label =\textup{(\arabic*)}]
\item $\mathcal{S}$ has an amalgam in $\OM$.
\item $\mathcal{S}$ has an amalgam in $\IdOM$.
\item $\mathcal{S}$ is compatible.
\end{enumerate}
\end{thm}

So to check whether a span of idempotent ordered monoids has an amalgam in the class of ordered monoids it suffices to check if it restricts (up to permuting the embeddings in the span) to one of the spans $\mathcal{S}_1$ and $\mathcal{S}_2$.  The next proposition shows that the spans $\mathcal{S}_1$ and $\mathcal{S}_2$ do also not have an amalgam in the bigger classes $\Sem\DLM$ and $\Id$.

\begin{prop}\label{p:bad span Id and Sem}
The spans $\mathcal{S}_1$ and $\mathcal{S}_2$ do not have an amalgam in the varieties  $\Sem\DLM$ and $\Id$.
\end{prop}
\begin{proof}
For $\Sem\DLM$ suppose for a contradiction that the span $\mathcal{S}_1$ has an amalgam $\mathbf{E} \in\Sem\DLM$ with inclusion maps $\mathbf{G}_3 \hookrightarrow \mathbf{E}$ and $\mathbf{D}_3 \hookrightarrow \mathbf{E}$. Then, since $\mathbf{E}$ is semilinear, we have without loss of generality $\mathbf{E} = \prod_{i\in I} \mathbf{E}_i$ for a family of ordered monoids $\set{\mathbf{E}_i}_{i\in I}$. So there is an $i\in I$ and there are homomorphisms $j_1\colon \mathbf{G}_3 \to \mathbf{E}_i$ and $j_2 \colon \mathbf{D}_3 \to \mathbf{E}_i$ such that $j_1\!\!\restriction_{C_2} = j_2\!\!\restriction_{C_2}$ and $j_1(e)  \neq j_1(\bot)$, so also $j_2(e) \neq j_2(\bot)$. But, since $\mathbf{G}_3$ and $\mathbf{D}_3$ are simple, $j_1$ and $j_2$ are embeddings. Hence $\mathbf{E}_i$ is an amalgam for the span which is an ordered monoid, contradicting \Cref{prop:not compatible no amalgam}. For the span $\mathcal{S}_2$ the claim follows by duality.

For $\Id$ suppose again for a contradiction that the span $\mathcal{S}_1$ has an amalgam $\mathbf{E} \in \Id$ with embeddings $j_1\colon \mathbf{G}_3 \hookrightarrow \mathbf{E}$ and $j_2\colon\mathbf{D}_3 \hookrightarrow \mathbf{E}$ and $\mathbf{E} \in \Id$. Then $j_1(e)\leq j_1(\top), j_2(\top^\ast)$, so 
\[
j_1(\top)\cdot j_2(\top^\ast) = j_1(\top)\join j_2(\top^\ast) =  j_2(\top^\ast)\cdot j_1(\top),
\]
 by \Cref{lemma:Idemp} (ii). Thus we get
 \[
 j_1(\top)\cdot j_2(\top^\ast)\cdot  j_2(\bot) =  j_2(\top^\ast)\cdot j_1(\top)\cdot j_1(\bot) = \ j_2(\top^\ast)\cdot j_1(\top).
 \]
  But we also have 
 \begin{align*}
  j_1(\top)\cdot j_2(\top^\ast) \cdot j_2(\bot) = j_1(\top)\cdot j_2(\bot) = j_1(\top)\cdot j_1(\bot)  = j_1(\top).
 \end{align*}
Hence $j_1(\top)\join j_2(\top^\ast) = j_1(\top) $, i.e., $j_2(\top^\ast) \leq j_1(\top)$, and as in the proof of \Cref{prop:not compatible no amalgam} we get a contradiction. For the span $\mathcal{S}_2$ the claim follows by duality.
\end{proof}

\begin{cor}
No subvariety of $\Sem\DLM$ or $\Id$ that contains the algebras $\mathbf{G}_3$ and $\mathbf{D}_3$ has the amalgamation property. In particular, the varieties $\Sem\DLM$, $\Sem\Id$, and $\Id$ do not have the amalgamation property.
\end{cor}
The next example shows that the spans $\mathcal{S}_1$ and $\mathcal{S}_2$ have  amalgams in the variety $\DLM$.
\begin{exmp}
Consider the chain $\mathbf{3} = \alg{\set{0,1,2}, \leq}$ with the natural order $0<1<2$ and the distributive $\ell$-monoid $\mathbf{End}(\mathbf{3}) = \alg{\End(\mathbf{3}),\meet,\join,\circ, id_3}$, where we recall that $\End(\mathbf{3})$ is the set of order-preserving maps on $\mathbf{3}$, $\circ$ is composition, and meet and join are defined point-wise. Let $\pair{k_0,k_1,k_2}$ denote the member of $\End(\mathbf{3})$ that maps $0$ to $k_0$, $1$ to $k_1$, and $2$ to $k_2$, e.g., $\pair{0,1,2}$ denotes $id_3$. We define $j_1 \colon \mathbf{G}_3 \to \mathbf{End}(\mathbf{3})$ by $j_1(\bot) = \pair{0,0,2}$, $j_1(e) = id_3$, $j_1(\top) = \pair{1,1,2}$ and $j_2\colon \mathbf{D}_3 \to\mathbf{End}(\mathbf{3})$ by $j_2(\bot) = \pair{0,0,2}$, $j_2(e) = id_3$, $j_2(\top^\ast) = \pair{0,2,2}$. 
Then $\pair{\mathbf{End}(\mathbf{3}),j_1,j_2}$ is an amalgam of the span $\mathcal{S}_1$ in $\DLM$. To see this first note that clearly both maps are order-preserving and all elements in the images of $j_1$ and $j_2$ are idempotent. Hence, it remains to show that the maps preserve the products of $\bot, \top$ and $\bot,\top^\ast$, respectively.  For this we calculate:
\begin{align*}
j_1(\top) \circ j_1(\bot) &= \pair{1,1,2} \circ \pair{0,0,2} = \pair{1,1,2} =  j_1(\top) = j_1(\top\bot), \\
j_1(\bot) \circ j_1(\top) &= \pair{0,0,2} \circ \pair{1,1,2} = \pair{0,0,2} =  j_1(\bot) = j_1(\bot\top), \\
j_2(\top^\ast) \circ j_2(\bot) &= \pair{0,2,2} \circ \pair{0,0,2} = \pair{0,0,2} =  j_2(\bot) = j_2(\top^\ast\bot), \\
j_2(\bot) \circ j_2(\top) &= \pair{0,0,2} \circ \pair{0,2,2} = \pair{0,2,2} =  j_2(\top^\ast) = j_2(\bot\top^\ast). 
\end{align*}
Therefore,  $\pair{\mathbf{End}(\mathbf{3}),j_1,j_2}$  is an amalgam for $\mathcal{S}_1$ in $\DLM$ and, by duality, also $\mathcal{S}_2$ has an amalgam in $\DLM$.
\end{exmp}

An algebra $\mathbf{A}$ is said to have the \emph{congruence extension property} if for every subalgebra $\bf B$ of $\bf A$ and congruence $\Psi \in \con(\mathbf{B})$ there exists  a congruence $\Theta \in \con(\mathbf{A})$ such that $\Psi = B^2 \cap \Theta$. 
We say that a variety $\var{V}$ has the \emph{congruence extension property} if all of its members have the congruence extension property.

\begin{prop}\label{prop:CEP}
The  algebras  $\mathbf{C}_4$, $\mathbf{C}_4^\partial$, $\mathbf{M} \boxplus \mathbf{N}$, and $\mathbf{N} \boxplus \mathbf{M}$ for $\mathbf{M} \in \set{\mathbf{C}_2,\mathbf{C}_2^\partial}$, $\mathbf{N} \in \set{\mathbf{G}_3,\mathbf{D}_3}$ do not have the congruence extension property.
\begin{proof}
We first prove that $\mathbf{C}_4^\partial$ does not have the congruence extension property (the proof for $\mathbf{C}_4$ is dual). Recall that the algebra $\mathbf{C}_4^\partial = \alg{\set{2,e,1,3},\cdot,e, \leq}$ with $2<e<1<3$. Consider the subalgebra $\mathbf{M} = \Sg(\set{e,1,3})$ of $\mathbf{C}_4^\partial$. The order on~$\bf M$ is $e<1<3$.
Clearly the relation $\Psi = \Delta_{M} \cup \set{\pair{1,3},\pair{3,1}}$ is a non-trivial congruence of $\bf M$. Let $\Theta$ be the congruence of $\mathbf{C}_4^\partial$ generated by $\Psi$. Then we have $\pair{2,3} = \pair{2\cdot 1, 2 \cdot 3} \in \Theta$. Hence, by the convexity of the congruence classes,   we get  $\Theta = {C}_4^\partial\times {C}_4^\partial $. 

For $\mathbf{C}_2 \boxplus \mathbf{G}_3$ with the bottom element of $\mathbf{C}_2$ renamed to $1$  consider the subalgebra $\mathbf{M} = \Sg(1,\bot,e)$. The order on $\mathbf{M}$ is $1 < \bot < e$ and the relation $\Psi = \Delta_2\cup\set{\pair{1,\bot},\pair{\bot,1}}$ is a congruence on $\mathbf{M}$. Let $\Theta$ be the congruence on  $\mathbf{C}_2 \boxplus \mathbf{G}_3$ generated by $\Psi$. Then we have  $\pair{1,\top} = \pair{\top\cdot 1, \top \cdot \bot} \in \Theta$. Thus, by the convexity of the congruence classes, we get $\Theta = (C_2 \cup G_3)\times (C_2 \cup G_3)$. The proofs for the other cases are very similar.
\end{proof}
\end{prop}

\begin{cor}
$\DLM$, $\Id$, $\Sem\DLM$, $\Sem\Id$, and $\Com\Id$ do not have the congruence extension property.
\end{cor}

\begin{lemma}\label{l:CEP CD locally finite}
A locally finite congruence-distributive variety $\var{V}$ has the congruence extension property if and only if all finite subdirectly irreducible algebras in $\var{V}$ have the congruence extension property.
\end{lemma}
\begin{proof}
By {\cite[Proposition 21]{Metcalfe2014}},  $\var{V}$ has the congruence extension property if and only if all finitely generated algebras in $\var{V}$ have the congruence extension property. Since we assume $\var{V}$ to be locally finite, this is equivalent to all finite algebras in $\var{V}$ having the congruence extension property.  But all finite algebras in $\var{V}$ embed into finite products of finite subdirectly irreducible algebras and, since $\var{V}$ is congruence-distributive, it follows, by \cite[Proposition 3.2]{Kiss1983}, that the congruence extension property is preserved by finite products. So $\var{V}$ having the congruence extension property is equivalent to all  finite subdirectly irreducible algebras in $\var{V}$ having the congruence extension property.
\end{proof}

\begin{prop}\label{l:CEP for small}
Let $\var{V}$ be a subvariety of $\Sem\Id$. Then $\var{V}$ has the congruence extension property if and only if  $\var{V}$ is a  subvariety of $V(\mathbf{C}_3,\mathbf{C}_3^\partial,\mathbf{G}_3,\mathbf{D}_3)$.
\end{prop}

\begin{proof}
For the right-to-left direction note that up to isomorphism the subdirectly irreducible members of $V(\mathbf{C}_3,\mathbf{C}_3^\partial,\mathbf{G}_3,\mathbf{D}_3)$ are $\mathbf{C}_2$, $\mathbf{C}_2^\partial$, $\mathbf{C}_3$, $\mathbf{C}_3^\partial$, $\mathbf{G}_3$, and $\mathbf{D}_3$ which clearly all have the congruence extension property. So, by  \Cref{l:CEP CD locally finite}, we get that $V(\mathbf{C}_3,\mathbf{C}_3^\partial,\mathbf{G}_3,\mathbf{D}_3)$  has the congruence extension property. The left-to-right direction follows from \Cref{prop:CEP} together with the fact that, by \Cref{thm:subdirect-equivalence} and \Cref{cor:counting subd},  any bigger finite subdirectly irreducible algebra in $\Sem\Id$ contains one of the algebras from \Cref{prop:CEP} as a subalgebra. 
\end{proof}

\begin{cor}
Let $\var{V}$ be a subvariety of $\Com\Id$. Then $\var{V}$ has the congruence extension property if and only if  $\var{V}$ is a  subvariety of $V(\mathbf{C}_3,\mathbf{C}_3^\partial)$.
\end{cor}

We call a variety $\var{V}$ \emph{residually small} if there is an upper bound on the cardinality of the universes of the subdirectly irreducible algebras in $\var{V}$. 
\begin{lemma}[{\cite[Corollary 2.11]{Kearnes1989}}]\label{l:AP and RS implies CEP}
Let $\var{V}$ be a congruence-distributive variety. If $\var{V}$ is residually small and has the  amalgamation property, then it has the congruence extension property.
\end{lemma}
Recall that a variety is called \emph{finitely generated} if it is generated by a finite set of finite algebras.
\begin{prop}\label{prop:amalgam-c4}
No finitely generated subvariety of $\DLM$ that contains any of the algebras. $\mathbf{C}_4$, $\mathbf{C}_4^\partial$, $\mathbf{M} \boxplus \mathbf{N}$, or $\mathbf{N} \boxplus \mathbf{M}$ for $\mathbf{M} \in \set{\mathbf{C}_2,\mathbf{C}_2^\partial}$, $\mathbf{N} \in \set{\mathbf{G}_3,\mathbf{D}_3}$  has the amalgamation property.
\end{prop}

\begin{proof}
Let $\var{V}$ be a finitely generated subvariety of $\DLM$ that contains one of the algebras. Then, by \Cref{l:CEP for small}, $\var{V}$ does not have the congruence extension property. Moreover, since $\var{V}$ is finitely generated, it follows from \Cref{l:jonsson} that $\var{V}_\ast$ contains only finitely many subdirectly irreducible algebras, yielding that it is residually small. Thus the claim follows from \Cref{l:AP and RS implies CEP}.
\end{proof}

\begin{cor}
No proper subvariety of $\Com\Id$ that contains $\mathbf{C}_4$ or $\mathbf{C}_4^\partial$ has the amalgamation property.
\end{cor}

Let  $\var{V}$ be a variety. We denote by $\var{V}_{\mathrm{FSI}}$ the class of finitely subdirectly irreducible members of $\var{V}$, i.e., $\var{V}_{\mathrm{FSI}}$ contains all algebras $\mathbf{A}\in \var{V}$ such that $\Delta_A$ is meet-irreducible in $\Con(\mathbf{A})$.  
Moreover, we say that a class of algebras $\var{K}$ has the \emph{one-sided amalgamation property} if for every span $\pair{i_1\colon \mathbf{A} \to \mathbf{B}, i_2\colon \mathbf{A} \to \mathbf{C}}$ in $\var{V}$ there exists an algebra $\mathbf{D}\in \var{K}$, a homomorphism $j_1 \colon \mathbf{B} \to \mathbf{D}$, and an embedding $j_2 \colon \mathbf{C} \to \mathbf{D}$ such that $j_1\circ i_1  = j_2\circ i_2$. The triple $\pair{\mathbf{D},j_1,j_2}$ is called a \emph{one-sided amalgam} of the span. We note that the amalgamation property implies the one-sided amalgamation property.

\begin{lemma}[Relativized  Jónsson's Lemma {\cite[Lemma 1.5.]{Czelakowski1990}}]\label{l:relat-jonsson}
Let $\var{K}$ be a class of algebras. Then $Q(\var{K})_{\mathrm{FSI}} \subseteq ISP_U(\var{K})$.
In particular,  if $\var{K}$ is a finite set of finite algebras and at least one member of $\var{K}$ has the trivial algebra as a subalgebra, then  $Q(\var{K})_{\mathrm{FSI}} \subseteq IS(\var{K})$.
\end{lemma}

\begin{thm}[{\cite[Corollary 3.5.]{Fussner2022}}]\label{thm:fin-subd-amalg}
Let $\var{V}$ be a congruence-distributive variety with the congruence extension property such that $\var{V}_{\mathrm{FSI}}$ is closed under subalgebras. Then $\var{V}$ has the amalgamation property if and only if $\var{V}_{\mathrm{FSI}}$ has the one-sided amalgamation property. 
\end{thm}

\begin{lemma}\label{l:IS-FSI}
For $n\in \set{2,3}$ we have  $V(\mathbf{C}_n)_{\mathrm{FSI}} = IS(\mathbf{C}_n)$, $V(\mathbf{C}_n^\partial)_{\mathrm{FSI}} = IS(\mathbf{C}_n^\partial)$,   $V(\mathbf{C}_n,\mathbf{C}_n^\partial)_{\mathrm{FSI}} = IS( \mathbf{C}_n,\mathbf{C}_n^\partial)$, $V(\mathbf{G}_3)_{\mathrm{FSI}} = IS(\mathbf{G}_3)$, $V(\mathbf{D}_3)_{\mathrm{FSI}} = IS(\mathbf{D}_3)$, and $V(\mathbf{G}_3, \mathbf{D}_3)_{\mathrm{FSI}} = IS(\mathbf{G}_3, \mathbf{D}_3)$.
\end{lemma}
\begin{proof}
For the left-to-right inclusions note that all the varieties are generated by the respective algebras as quasi-varieties, by \Cref{l:jonsson} and \Cref{l:IS=HS}. Thus the left-to-right inclusion follows from \Cref{l:relat-jonsson}.

For the right-to-left inclusions note that for $n\in \set{2,3}$ every member of $IS( \mathbf{C}_n,\mathbf{C}_n^\partial)$ is either trivial or subdirectly irreducible. The same is also true  for  $V(\mathbf{G}_3)$, $V(\mathbf{D}_3)$ and $V(\mathbf{G}_3, \mathbf{D}_3)$.
\end{proof}

\begin{prop}\label{prop:amalgam-c3-join}
None of the  varieties $V(\mathbf{M},\mathbf{N})$ generated by  distinct algebras $\mathbf{M}, \mathbf{N} \in \set{\mathbf{C}_3,\mathbf{C}_3^\partial,\mathbf{G}_3,\mathbf{D}_3}$   has the amalgamation property.
\end{prop}

\begin{proof}
For $V(\mathbf{G}_3,\mathbf{D}_3)$ the claim is immediate from \Cref{p:bad span Id and Sem}.
Otherwise note that, by \Cref{l:CEP for small}, in any case the variety has the congruence extension property and, by \Cref{l:IS-FSI}, we have that $V(\mathbf{M},\mathbf{N})_{\mathrm{FSI}} = IS(\mathbf{M},\mathbf{N})$ is closed under subalgebras. 

For  $V(\mathbf{C}_3,\mathbf{C}_3^\partial)$  we have that $IS(\mathbf{C}_3,\mathbf{C}_3^\partial)$ consists up to isomorphism of the algebras $\boldsymbol{0}$, $\mathbf{C}_2$, $\mathbf{C}_2^\partial,\mathbf{C}_3$, and~$\mathbf{C}_3^\partial$.   
Suppose for a contradiction that the variety $V(\mathbf{C}_3,\mathbf{C}_3^\partial)$ has the amalgamation property and consider  the span $\pair{i_1 \colon \mathbf{C}_2^\partial \hookrightarrow \mathbf{C}_3, i_2\colon  \mathbf{C}_2^\partial \hookrightarrow \mathbf{C}_3^\partial }$.  Then it follows from \Cref{thm:fin-subd-amalg} that there exists an algebra $\mathbf{D} \in \set{\boldsymbol{0}, \mathbf{C}_2,\mathbf{C}_2^\partial,\mathbf{C}_3, \mathbf{C}_3^\partial}$, a homomorphism $j_1 \colon \mathbf{C}_3 \to \mathbf{D}$, and an embedding $j_2 \colon \mathbf{C}_3^\partial \to \mathbf{D}$ such that $j_1\circ i_1 = j_2 \circ i_2$. 
Since $j_2$ is an embedding, we must have $\mathbf{D} = \mathbf{C}_3^\partial$ and $j_2 = id_{C_3^\partial}$. If we use the notation $C_2^\partial = \set{e < \top}$, $C_3 = \set{2 < e < 1}$ and $C_3^\partial = \set{1^\partial < e < 2^\partial}$, we get
\begin{equation*}
2^\partial = j_2(2^\partial ) = j_2(i_2(\top)) = j_1(i_1(\top)) = j_1(1) .
\end{equation*}
Thus we also  get
\begin{equation*}
 j_1(2) = j_1(2 \cdot 1) = j_1(2)j_1(1) = j_1(2)\cdot  2^\partial = 2^\partial,
\end{equation*}
yielding $j_1(e) = e < 2^\partial =  j_1(2)$, contradicting the fact that homomorphisms are order-preserving.

For $V(\mathbf{M},\mathbf{N})$ with $\mathbf{M} \in \set{\mathbf{G}_3,\mathbf{D}_3}$ and  $\mathbf{N} \in \set{\mathbf{C}_3,\mathbf{C}_3^\partial}$  we note that the class $IS(\mathbf{M},\mathbf{N})$ consists up to isomorphisms of the algebras  $\boldsymbol{0}$, $\mathbf{C}_2$, $\mathbf{C}_2^\partial,\mathbf{M}$, and~$\mathbf{N}$.   Suppose for a contradiction that  $V(\mathbf{M},\mathbf{N})$ has the amalgamation property and consider the span $\pair{i_1 \colon \mathbf{C}_2 \hookrightarrow \mathbf{M},i_2\colon \mathbf{C}_2 \hookrightarrow \mathbf{N}}$. Then, by \Cref{thm:fin-subd-amalg}, there exists an algebra $\mathbf{D} \in \set{\boldsymbol{0}, \mathbf{C}_2,\mathbf{C}_2^\partial,\mathbf{M}, \mathbf{N}}$, a homomorphism $j_1 \colon \mathbf{M} \to \mathbf{D}$, and an embedding $j_2 \colon \mathbf{N} \to \mathbf{D}$ such that $j_1\circ i_1 = j_2 \circ i_2$. Because $j_2$ is an embedding we get $\mathbf{D} = \mathbf{N}$, by cardinality reasons. But also, since $\mathbf{M}$ is simple and not isomorphic to $\mathbf{N}$,  the map $j_1$ needs to be constant with $j_2(M) = \set{e}$, contradicting $j_1\circ i_1 = j_2 \circ i_2$.
\end{proof}

\begin{prop}\label{prop:amalgam-small}
The varieties $V(\mathbf{C}_2)$, $V(\mathbf{C}_2^\partial)$,  $V(\mathbf{C}_2,\mathbf{C}_2^\partial)$, $V(\mathbf{C}_3)$,  $V(\mathbf{C}_3^\partial)$, $V(\mathbf{G}_3)$, and $V(\mathbf{D}_3)$ have the amalgamation property.
\end{prop}

\begin{proof}
To shorten the notation we let $\var{U} = V(\mathbf{C}_2)$, $\var{V} = V(\mathbf{C}_2,\mathbf{C}_2^\partial)$, and $\var{W} = V(\mathbf{C}_3)$. By \Cref{l:IS-FSI}, $\var{U}_{\mathrm{FSI}} = IS(\mathbf{C}_2)$, $\var{V}_{\mathrm{FSI}} = IS (\mathbf{C}_2,\mathbf{C}_2^\partial)$, and $\var{W}_{\mathrm{FSI}} = IS(\mathbf{C}_3)$ are closed under subalgebras. So, by \Cref{thm:fin-subd-amalg}, it suffices to show that $\var{U}_{\mathrm{FSI}}$, $\var{V}_{\mathrm{FSI}}$, and $\var{W}_{\mathrm{FSI}}$ have the one-sided amalgamation property.

For $\var{U}$ note that $\var{U}_{\mathrm{FSI}} =  IS(\mathbf{C}_2)$ consists up to isomorphism of the algebras $\mathbf{C}_2$ and $\boldsymbol{0}$.  Thus every span in $\var{U}_{\mathrm{FSI}}$ has $\mathbf{C}_2$ with the respective inclusion maps as a one-sided-amalgam. So $\var{U}_{\mathrm{FSI}}$ has the amalgamation property.

For  $\var{V}$ note that up to isomorphism $\var{V}_{\mathrm{FSI}}=  IS(\mathbf{C}_2,\mathbf{C}_2^\partial)$ consists of the algebras $\boldsymbol{0}$, $\mathbf{C}_2$, and $\mathbf{C}_2^\partial$. Thus  the only spans in~$\var{V}_{\mathrm{FSI}}$ for which a one-sided amalgam is not clear are the spans $\pair{i_1,i_2}$ and $\pair{i_2,i_1}$, where $i_1 \colon \boldsymbol{0} \hookrightarrow \mathbf{C}_2$ and $i_2 \colon \boldsymbol{0}\hookrightarrow \mathbf{C}_2^\partial$ are the obvious embeddings. But  for the span $\pair{i_1,i_2}$ the maps $f\colon \mathbf{C}_2 \to \mathbf{C}_2^\partial$, $f(e) = f(\bot) = e$ and $id_{C_3^\partial}$ form a one-sided amalgam, and for the span $\pair{i_2,i_1}$ the maps $id_{C_2}$ and $g\colon \mathbf{C}_2^\partial \to \mathbf{C}_2$, $g(e) = g(\top) = e$ form a one-sided amalgam. Thus $\var{V}_{\mathrm{FSI}}$ has the one-sided amalgamation property.

For $\var{W}$ again note that $\var{W}_{\mathrm{FSI}} =  IS(\mathbf{C}_3)$ consists up to isomorphism of  the algebras $\boldsymbol{0}$, $\mathbf{C}_2$, $\mathbf{C}_2^\partial$, and $\mathbf{C}_3$. Let  $\pair{i_1\colon \mathbf{A} \hookrightarrow \mathbf{B},i_2\colon \mathbf{A}\hookrightarrow \mathbf{C}}$ be a span in $\var{W}_{\mathrm{FSI}}$, i.e., without loss of generality $\mathbf{A},\mathbf{B},\mathbf{C} \in \set{\boldsymbol{0}, \mathbf{C}_2,\mathbf{C}_2^\partial,\mathbf{C}_3}$. Then  the inclusion maps $j_1 \colon \mathbf{B} \hookrightarrow \mathbf{C}_3$ and $j_2\colon \mathbf C \hookrightarrow \mathbf{C}_3$ form an amalgam, since for every $\mathbf{A} \in \set{\boldsymbol{0}, \mathbf{C}_2,\mathbf{C}_2^\partial,\mathbf{C}_3} $ there exists a unique embedding into $\mathbf{C}_3$. Hence~$\var{W}_{\mathrm{FSI}}$ has the amalgamation property.

For $V(\mathbf{C}_2^\partial)$  and  $V(\mathbf{C}_3^\partial)$ the claim follows by duality and for   $V(\mathbf{D}_3)$ and $V(\mathbf{G}_3)$ the proofs are very similar to  the proof  for $V(\mathbf{C}_3)$, also using the uniqueness of embeddings.
\end{proof}

Combining Propositions~\ref{prop:amalgam-c4}, \ref{prop:amalgam-c3-join}, and \ref{prop:amalgam-small} we get the following characterization for amalgamation in non-trivial finitely generated subvarieties of $\Sem\Id$. 

\begin{thm}\label{thm:subvariety amalgamation}
The non-trivial finitely generated subvarieties of $\Sem\Id$  that have the amalgamation property are $V(\mathbf{C}_2)$, $V(\mathbf{C}_2^\partial)$,  $V(\mathbf{C}_2, \mathbf{C}_2^\partial)$, $V(\mathbf{C}_3)$, $V(\mathbf{C}_3^\partial)$, $V(\mathbf{G}_3)$, and $V(\mathbf{D}_3)$.
\end{thm}

\begin{cor}
The non-trivial proper subvarieties of $\Com\Id$ with the amalgamation property are $V(\mathbf{C}_2)$, $V(\mathbf{C}_2^\partial)$,  $V(\mathbf{C}_2, \mathbf{C}_2^\partial)$, $V(\mathbf{C}_3)$, and $V(\mathbf{C}_3^\partial)$.
\end{cor}

Let us define for $n\in \mathbb{N}$ the algebras 
\begin{equation*}
\mathbf{G}_{2n+1} \coloneqq \bigboxplus_{i=1}^n \mathbf{G}_3 \quad \text{and} \quad \mathbf{D}_{2n+1} \coloneqq \bigboxplus_{i=1}^n \mathbf{D}_3.
\end{equation*}
Note that $\mathbf{G}_1 = \mathbf{D}_1 = \mathbf{0}$.
Moreover,  for  $n>1$ we assume that  $G_{2n+1} = \set{\bot_1,\dots, \bot_n,e,\top_n,\dots, \top_1}$, where for $i\in \set{1,\dots, n}$ the elements $\bot_i$ and $\top_i$ are the bottom and top element of the $i$-th copy of $\mathbf{G}_3$, respectively. Analogously we assume that for $n>1$ we have  $D_{2n+1} = \set{\bot_1,\dots, \bot_n,e,\top_n,\dots, \top_1}$.

\begin{lemma}\label{l:non-fin-gen-comm}
If $\var{V}$ is a subvariety of $\Sem\Id$ that is not finitely generated, then $\Com\Id \subseteq \var{V}$.
\end{lemma}
\begin{proof}
Suppose that $\var{V}$ is a subvariety of $\Sem\Id$ that is not finitely generated. Then, since $\var{V}$ is locally finite, there exists an indexed family of finite subdirectly irreducible algebras $\set{\mathbf{M}_i}_{i \in \mathbb{N}}$ with $\lvert M_i \rvert < \lvert M_{i+1} \rvert$ for each $i\in \mathbb{N}$, i.e., the family contains finite subdirectly irreducible algebras which strictly increase in size. By \Cref{cor:subdirect irred in Sem}, each subdirectly irreducible member of $\var{V}$ is totally ordered. So let $c(\mathbf{M}_i)$ be the number of components in the nested sum decomposition of $\mathbf{M}_i$. Note that, since the biggest possible components are $\mathbf{G}_3$ and $\mathbf{D}_3$, we have $\lvert M_i \rvert \leq 1 + 2 c(\mathbf{M}_i)$. Therefore, without loss of generality we may also assume that also $c(\mathbf{M}_i) < c(\mathbf{M}_{i+1})$ for each $i\in \mathbb{N}$.  Thus, by \Cref{thm:subdirect-equivalence} and the fact that $\mathbf{C}_2,\mathbf{C}_2^\partial \leq_{IS} \mathbf{G}_3,\mathbf{D}_3$, for every $n\geq 1$ there exists an $i\in \mathbb{N}$ such that $\mathbf{C}_n \leq_{IS} \mathbf{M}_i$. Hence, $\Com\Id \subseteq \var{V}$, by \Cref{thm:CId lattice}.
\end{proof}

\begin{prop}
Let $\var{V}$ be a subvariety of $\Sem\Id$. If $\var{V}$ is not finitely generated and each span of totally ordered members has an amalgam in $\var{V}$, then $\var{V} \in \set{\Com\Id, V(\set{\mathbf{G}_{2n+1}\mid n\in \mathbb{N}}), V(\set{\mathbf{D}_{2n+1}\mid n\in \mathbb{N}})}$.
\end{prop}

\begin{proof}
Since $\var{V}$ is not finitely generated, by \Cref{l:non-fin-gen-comm}, we have $\Com\Id \subseteq \var{V}$. If $\var{V} \neq \Com\Id$, then not all members of $\var{V}$ are commutative. Hence either $\mathbf{G}_3 \in \var{V}$ or $\mathbf{D}_3 \in \var{V}$, but not both, by \Cref{p:bad span Id and Sem}. Assume, without loss of generality, that $\mathbf{G}_3 \in \var{V}$. Then every finite subdirectly irreducible member of $\var{V}$ is isomorphic to a finite nested sum of the algebras $\mathbf{C}_2,\mathbf{C}_2^\partial,\mathbf{G}_3$, so clearly $\var{V} \subseteq V(\set{\mathbf{G}_{2n+1}\mid n\in \mathbb{N}})$, by \Cref{l:embeddings to summands}. For the reverse inclusion we prove by induction that for every $n\in \mathbb{N}$ we have $\mathbf{G}_{2n+1} \in \var{V}$. Before we start the induction note that, since $\Com\Id \subseteq \var{V}$, we have $\mathbf{C}_3 \in \var{V}$ and, by assumption, the span $\pair{i_1 \colon \mathbf{C}_2^\partial \hookrightarrow \mathbf{C}_3, i_2\colon \mathbf{C}_2^\partial \hookrightarrow \mathbf{G}_3}$ has an amalgam $\pair{\mathbf{A},j_1\colon \mathbf{C}_3 \to \mathbf{A}, j_2 \colon \mathbf{G}_3 \to \mathbf{A}}$ in $\var{V}$. Moreover, using the notation $C_3 = \set{2<e<1}$ we get
\begin{equation*}
 j_1(2)  \cdot  j_2(\bot) =   j_1(2) \cdot j_1(1) \cdot j_2(\bot) =  j_1(2) \cdot j_2(\top) \cdot j_2(\bot) = j_1(2)\cdot j_2(\top) = j_1(2).
\end{equation*} 
But, since $j_1(2),j_2(\bot) \leq e$, by \Cref{lemma:Idemp}, we have $j_1(2) \cdot j_2(\bot) = j_1(2) \meet j_2(\bot)$, yielding $j_1(2) \leq j_2(\bot)$. Moreover, since  $j_1(2)$ commutes with $j_1(1) = j_2(\top)$ and $j_2(\bot)$ does not commute with $j_2(\top)$, we have $j_1(2) \neq j_2(\bot)$, yielding $j_1(2) < j_2(\bot)$. Summarizing, we have $j_1(2) < j_2(\bot) < e < j_2(\top)$ and $j_1(2)$ absorbs any element in $\set{j_2(\bot),e,j_2(\top)}$. Thus,  $\mathbf{A}$ contains a copy of $\mathbf{C}_2 \boxplus \mathbf{G}_3$  as a subalgebra. Hence $\mathbf{C}_2 \boxplus \mathbf{G}_3 \in \var{V}$. 

The base case $n=1$ of the induction is clear. Suppose that $\mathbf{G}_{2n+1} \in \var{V}$ and  that $\mathbf{C}_2 \boxplus \mathbf{G}_3$ has the universe $\set{1 < \bot < e < \top}$. Now consider the span $\pair{i_1\colon \mathbf{C}_2 \to \mathbf{C}_2 \boxplus \mathbf{G}_3, i_2\colon  \mathbf{C}_2 \to \mathbf{G}_{2n+1}}$ with $i_1(\bot) = 1$ and $i_2(\bot) = \bot_n$.  Then, by assumption, we know that  the span has an amalgam $\pair{\mathbf{B},j_1\colon \mathbf{C}_2 \boxplus \mathbf{G}_3 \to \mathbf{B}, j_2\colon \mathbf{G}_{2n+1} \to \mathbf{B}}$. Since $1\cdot \bot = 1\cdot \top = 1$ in $\mathbf{C}_2 \boxplus \mathbf{G}_3$ and $j_2(\bot_n) = j_1(1)$, we get $j_2(\bot_n) \cdot  j_1(\bot) = j_2(\bot_n)$. Thus as above $j_2(\bot_n) < j_1(\bot)$, since $j_2(\bot_n),j_1(\bot) \leq e$ and $j_2(\bot_n) = j_1(1)$ commutes with $j_1(\top)$. On the other hand we have
\begin{equation*}
j_2(\top_n)\cdot j_2(\top) = j_2(\top_n)\cdot j_2(\bot_n) \cdot  j_1(\top) 
=  j_2(\top_n)\cdot j_2(\bot_n)
= j_2(\top_n).
\end{equation*}
But, since $e \leq j_2(\top_n), j_1(\top)$, by \Cref{lemma:Idemp}, $j_2(\top_n)\cdot j_1(\top) = j_2(\top_n) \join j_1(\top)$, yielding $j_1(\top) \leq j_2(\top_n)$. Moreover, since $j_1(\top)$ commutes with $j_2(\bot_n) = j_1(1)$ and $j_2(\top_n)$ does not, we get $j_2(\top_n) \neq j_1(\top)$, i.e., $j_2(\top) < j_1(\top_n)$.  Summarizing, we have $j_2(\bot_n) < j_1(\bot) < e < j_1(\top) < j_2(\top_n)$, constituting a copy of $\mathbf{G}_{5}$ in $\mathbf{A}$. But also, for every $k<n$, $\bot_k$ and $\top_k$ absorb $\bot_n$ and $\top_n$. Hence it follows from the above that $j_2(\bot_k)$ and $j_2(\top_k)$ absorb $j_1(\bot)$ and $j_1(\top)$. Moreover, we have 
\[
j_2(\bot_1) < \dots < j_2(\bot_n) < j_1(\bot) < e < j_1(\top) < j_2(\top_n) < \dots < j_1(\top_1).
\]
Thus, $\mathbf{B}$ contains a subalgebra isomorphic to $ \mathbf{G}_{2n+1} \boxplus \mathbf{G}_3 = \mathbf{G}_{2(n+1)+1}$ and we get that $\mathbf{G}_{2(n+1) + 1} \in \var{V}$. This concludes the proof.
\end{proof}

\begin{cor}\label{c:amalgamation infinite}
Let $\var{V}$ be a subvariety of $\Sem\Id$ that is not finitely generated such that $\var{V} \notin \set{\Com\Id, V(\set{\mathbf{G}_{2n+1}\mid n\in \mathbb{N}}), V(\set{\mathbf{D}_{2n+1}\mid n\in \mathbb{N}})}$. Then $\var{V}$ does not have the amalgamation property.
\end{cor}

Note that the class $\Com\IdOM$ of commutative idempotent ordered monoids contains all finitely subdirectly irreducible members of $\Com\Id$, is closed under subalgebras, and has the strong amalgamation property. Similarly the classes containing the totally ordered members of the varieties  $ V(\set{\mathbf{G}_{2n+1}\mid n\in \mathbb{N}})$ and $ V(\set{\mathbf{D}_{2n+1}\mid n\in \mathbb{N}})$, respectively, contain all finitely subdirectly irreducible members, are closed under subalgebras, and have the amalgamation property, by using the construction in the proof of \Cref{thm:compatible spans have strong amalgams}.
 But the varieties $\Com\Id$, $V(\set{\mathbf{G}_{2n+1}\mid n\in \mathbb{N}})$, and  $V(\set{\mathbf{D}_{2n+1}\mid n\in \mathbb{N}})$ do not have the congruence extension property, so there is no obvious way of extending the amalgamation property from their totally ordered members to the whole variety. Hence, it remains open whether these three varieties have the amalgamation property or not. 

\section*{Acknowledgment}
The author is grateful to George Metcalfe, Nick Galatos, and Wesley Fussner for several helpful comments and suggestions on earlier versions of this work, and to the anonymous referee for their comments and suggestions, which have helped to improve the paper.

%%%%%%%%%%%%%%%%%%%%%%%%%%%%%%%%%%%%%%%%%%%%%%
\bibliographystyle{plain}

\end{document}